\documentclass[11pt,reqno]{amsart}

\usepackage{amsthm}
\usepackage[usenames,dvipsnames]{pstricks}
\usepackage{epsfig}
\usepackage{pst-grad} 
\usepackage{pst-plot} 
\usepackage{graphicx,subfigure}
\usepackage{amsmath,amssymb}

\usepackage{acronym}
\acrodef{BO}{{\sl Benjamin-Ono}}
\acrodef{rBO}{{\sl regularized Benjamin-Ono}}
\acrodef{rILW}{{\sl regularized Intermediate Long Wave}}
\acrodef{DSW}{{\sl Dispersive Shock Wave}}
\acrodef{DSWs}{{\sl Dispersive Shock Waves}}
\acrodef{ILW}{{\sl Intermediate Long Wave}}
\acrodef{CGN}{{\sl Conjugate Gradient-Newton}}
\acrodef{SW/SW}{{\sl Shallow water / Shallow water}}
\acrodef{B/B}{{\sl Boussinesq / Boussinesq}}

\makeatletter
\newcommand{\sech}{\mathop{\operator@font sech}}
\newcommand{\sign}{\mathop{\operator@font sign}}
\makeatother

\newtheorem{lemma}{Lemma}[section]
\newtheorem{theorem}{Theorem}[section]

\newtheorem{remark}{Remark}[section]

\numberwithin{equation}{section}

\usepackage{multirow}
\usepackage{amssymb}
\usepackage{amsbsy}
\usepackage{amscd}
\usepackage{amsfonts}
\usepackage{amsmath}
\usepackage{amssymb}
\usepackage{amstext}
\usepackage{amsthm}
\usepackage{enumerate}
\usepackage{epsfig}
\usepackage{fancybox}
\usepackage{graphicx}
\usepackage{latexsym}
\usepackage{niceframe}
\usepackage{lipsum} 
\usepackage{tikz}
\usepackage{color}
\usepackage{dsfont}

\usepackage{epsfig}
\usepackage[english]{babel}
\usepackage[latin1]{inputenc}
\usepackage{indentfirst}
\usepackage{graphicx,subfigure}
\usepackage{tabularx}
\usepackage{xspace} 
\usepackage{amsmath,amssymb}
\usepackage{color}
\definecolor{oneblue}{rgb}{0,0.0,0.75}

\DeclareMathOperator*{\esssup}{ess\,sup}

\usepackage{amsthm}

\begin{document}
\title[Mathematical properties and numerical approximation...]{Mathematical properties and numerical approximation of pseudo-parabolic systems}

\author{E. Abreu}
\address{\textbf{E.~Abreu:} Universidade Estadual de Campinas,
13.083-970, Campinas, SP, Brazil}
\email{eabreu@ime.unicamp.br}

\author{E. Cuesta}
\address{\textbf{E.~Cuesta:} Applied Mathematics Department, University of Valladolid, P/ Belen 15, 47011, Valladolid, Spain}
\email{eduardo.cuesta@uva.es}

\author{A. Dur\'an}
\address{\textbf{A.~Dur\'an:} Applied Mathematics Department, University of Valladolid, P/ Belen 15, 47011, Valladolid, Spain}
\email{angeldm@uva.es}

\author{W. Lambert}
\address{\textbf{W.~Lambert:} Federal University of Alfenas, 
UNIFAL, Po\c cos de Caldas, MG, Brazil}
\email{wanderson.lambert@unifal-mg.edu.br}





\begin{abstract}
The paper is concerned with the mathematical theory and numerical approximation of systems of partial differential equations (pde) of hyperbolic, pseudo-parabolic type. Some mathematical properties of the initial-boundary-value problem (ibvp) with Dirichlet boundary conditions are first studied. They include the weak formulation, well-posedness and existence of traveling wave solutions connecting two states, when the equations are considered as a variant of a conservation law. Then, the numerical approximation consists of a spectral approximation in space based on Legendre polynomials along with a temporal discretization with strong stability preserving (SSP) property. The convergence of the semidiscrete approximation is proved under suitable regularity conditions on the data. The choice of the temporal discretization is justified in order to guarantee the stability of the full discretization when dealing with nonsmooth initial conditions.
A computational study explores the performance of the fully discrete scheme with regular and nonregular data.
\end{abstract}




\maketitle

\tableofcontents
\section{Introduction}
\label{sec1}
This paper is concerned with the theoretical and numerical analysis of pde systems of hyperbolic, pseudo-parabolic type. 
They are formulated as follows.
Let $d\geq 1$ be an integer, $T>0$. For $u=(u_{1},\ldots,u_{d})^{T}\in\mathbb{R}^{d}$, $x_{L}\leq x\leq x_{R}, 0\leq t\leq T$ the Dirichlet ivp considered here, in its general form, is
\begin{eqnarray}
&&\left(I-\frac{\partial}{\partial x}\left(A\frac{\partial}{\partial x}\right)\right)\frac{\partial}{\partial t}u
=-\frac{\partial}{\partial x}\left(B\frac{\partial}{\partial x}u\right)+\frac{\partial}{\partial x}G(u)+\gamma(u,x,t),\label{eq:3psystem1a}\\
&&u(x_{L},t)=g^{L}(t),\quad u(x_{R},t)=g^{R}(t),\quad0\leq t\leq T,\label{eq:3psystem1b}\\
&&u(x,0)=u_{0}(x),\quad x_{L}\leq x\leq x_{R}.\label{eq:3psystem1c}
\end{eqnarray}
The elements of (\ref{eq:3psystem1a})-(\ref{eq:3psystem1c}) are:
\begin{itemize}
\item[(H1)] $A=A(u)=\left(a_{ij}(u)\right)_{i,j=1}^{d}$ satisfies the conditions:
\begin{enumerate}
\item There exists $\alpha>0$ such that for all $u$
\begin{eqnarray}
\xi^{T}A(u)\xi\geq\alpha|\xi|^{2},\; \xi\in\mathbb{R}^{d}.\label{spd}
\end{eqnarray}
\item For $i,j=1,\ldots,d$, $a_{ij}=a_{ij}(u)$ is Lipschitz continuous with constant $L$.
\end{enumerate}
\item[(H2)] $B=B(u)=\left(b_{ij}(u)\right)_{i,j=1}^{d}, G=G(u)\in\mathbb{R}^{d}, \gamma=\gamma(x,t,u)\in\mathbb{R}^{d}$ are $C^{1}$, bounded and Lipschitz continuous in their arguments.
\end{itemize}
On the other hand, $g^{L},g^{R}:[0,T]\rightarrow \mathbb{R}^{d}$ are $C^{1}([0,T])$, with $u_{0}:[x_{L},x_{R}]\rightarrow\mathbb{R}^{d}$ as the initial condition.

For the sake of clarity, several simplifications in (\ref{eq:3psystem1a})-(\ref{eq:3psystem1c}) will be assumed throughout the present paper:
\begin{itemize}
\item[(i)] We take $x_{L}=-1, x_{R}=1$,  being aware of the corresponding scaling  when dealing with a different interval (see the numerical experiments in section \ref{sec4}).
\item[(ii)] The boundary conditions are homogeneous ($g^{L}=g^{R}=0$). Otherwise, defining the auxiliary function $\overline{u}=(u^{(1)},\ldots,u^{(d)})^{T}$ with
$$u^{(j)}(x,t)=\frac{g_{j}^{R}(t)-g_{j}^{L}(t)}{2}x+\frac{g_{j}^{R}(t)+g_{j}^{L}(t)}{2},\; j=1,\ldots,d,$$ then (\ref{eq:3psystem1a})-(\ref{eq:3psystem1c}) can be trasnformed to a homogeneous Dirichlet problem of the same type for $v=u-\overline{u}$.
\item[(iii)] We will assume that $\gamma$ is a function of $u$. The results obtained in the paper can be extended to the case of additional dependences on the variables $x$ and $t$ in the expected way.
\end{itemize}
From the seminal papers \cite{Sobolev1954,Ting1969}, 
hyperbolic, pseudo-parabolic systems of the form \eqref{eq:3psystem1a} cover a wide range of modelling of physical phenomena with pde's, mainly in fluid flow and heat conduction problems, \cite{Bear2018,barenblatt1960basic}.
The aims of the present paper are to provide a rigurous mathematical analysis to justify this type of pde's in modelling and to propose an efficient numerical method for the approximation. Of particular relevance for us is the alternative provided by these systems to model
multiphase flow problems in porous media,  traditionally studied with hyperbolic pde's. Here the pseudo-parabolic terms are associated to the additional assumption of nonequilibrium effects in the capillary pressure-saturation relationships, \cite{SHWG90,SHWG93a,SHWG93b,Baren2,Baren3,Juanes}, see also \cite{13DPP,CNP19,duijn2013travel,AbreuV2017,ABFL19,AFV20,KSMS11,EAPFWL23,Chen} and references therein. The study of the dynamics of nonequilibrium pseudo-parabolic two- and three-phase flow models is one of the main motivations of this work and will be further addressed in 
detail in \cite{ACDL2024} by using the results of the present paper.

The classical references for the mathematical analysis of the ivp for \eqref{eq:3psystem1a} and different ibvp's are \cite{ShowalterT1970,Showalter1978}. As for approximation methods, most of the literature on computational aspects and numerical analysis is concerned with the 1D scalar case, covering almost all the approximation tools for the spatial and temporal discretization (such as finite differences, finite elements, finite volumes, spectral methods, and discontinuous Galerkin methods) and with different techniques (hybrid, splitting/nonsplitting, etc). We refer \cite{EAAD20} for an exhaustive bibliography on the subject.

The main contributions of the present paper are the following:
\begin{itemize}
\item[(i)] 
The problem \eqref{eq:3psystem1a}-\eqref{eq:3psystem1c} is shown to admit an equivalent weak formulation. This is used to prove, under suitable hypotheses on the coefficients, the well-posedness in the form of a result of existence and uniqueness of solution of the weak form, as well as a regularity theorem.
\item[(ii)] The paper \cite{EAAD20} proposes to discretize the ibvp of a pseudo-parabolic equation with Dirichlet boundary conditions with a spectral approximation in space based on Jacobi polynomials and SSP methods for the time numerical integration. The semidiscrete Galerkin and collocation approximations are shown to exist and error estimates with respect to the exact solution are derived in suitable norms. These results are extended here for the case of the ibvp \eqref{eq:3psystem1a}-\eqref{eq:3psystem1c} in several ways. The first one is concerned with the pseudo-parabolic term which, compared to \cite{EAAD20}, can be nonlinear (in the sense that the matrix $A$ may now depend on $u$). The second point is the obvious extension to systems. As well, a spectral semidiscretization based on Legendre polynomials is analyzed. The family of Legendre polynomials is a subfamily of Jacobi polynomials which is a natural choice to discretize in space problems with Dirichlet boundary conditions.
In particular, the accuracy of the spectral semidiscretization given by the error estimates is established in terms of the degree of polynomial approximation and the regularity of the data of the problem. Some implementation details are provided and the extension of the results to other families of Jacobi polynomials will be discussed. 
\item[(iii)] The error estimates are valid under certain regularity conditions of the data. When some of them does not hold, a loss of global accuracy is expected in the form, amog other possible phenomena, of a reduction of order. This fact is illustrated and discussed by means of a computational study. In the particular case of discontinuous initial conditions, the use of temporal discretizations with high order of dispersion and strong stability preserving properties has shown to improve the performance of the numerical approximation, in the sense of reducing the errors in the expected oscillatory parts of the solution and controlling the stability close to the discontinuity.

\end{itemize}
The structure of the paper is as follows. In section \ref{sec2}, the weak version of \eqref{eq:3psystem1a}-\eqref{eq:3psystem1c} is formulated and well-posedness results are proved. Section \ref{sec3} is devoted to the numerical approximation of \eqref{eq:3psystem1a}-\eqref{eq:3psystem1c}. The Legendre-Galerkin spectral discretization is first analyzed: existence of semidiscrete solution and error estimates are proved. Further implementation details are provided. Then the choice of high-order dispersive and SSP time integrators is justified. The resulting fully discrete method is checked in accuracy and stability in section \ref{sec4} with a computational study, which includes numerical experiments with smooth and nonsmooth data. Some concluding remarks are in section \ref{sec5}. Finally  \ref{appen1} consists of a discussion on the existence of traveling waves solutions of the Riemann problem where \eqref{eq:3psystem1a} is considered as a diffusive-dispersive variant of a conservation law

The following notation will be used throughout the paper. 
For positive integer $p$, $L^{p}=L^{p}(\Omega)$ denotes 
the normed space of $L^{p}$-functions on $\Omega=(-1,1)$, while for nonnegative 
integer $m$, $C^{m}(\overline{\Omega})$ is the space of 
$m$-th order continuously differentiable functions on 
$\overline{\Omega}$. 
The standard inner product in $L^{2}$ will be denoted by $(\varphi,\psi)_{0}, \varphi, \psi\in L^{2}$, with associated norm given by $|\cdot |_{0}$.
For the Sobolev spaces, $H^{k}=H^{k}(\Omega), k\geq 0$ integer (where $H^{0}=L^{2}$), the corresponding 
norm will be denoted by
\begin{eqnarray*}
|\varphi |_{k}^{2}=\sum_{j=0}^{k}\left|\frac{d^{j}}{dx^{j}}\varphi\right|_{0}^{2}.\label{12b}
\end{eqnarray*}
We will also consider the spaces 
$H_{0}^{k}=H_{0}^{k}(\Omega)$ of functions 
$\varphi\in H^{k}$ such that $\varphi(-1)=\varphi(1)=0$. 
For $s\geq 0$, $H^{s}=H^{s}(\Omega)$ (and 
$H_{0}^{s}=H_{0}^{s}(\Omega)$) are defined by 
interpolation theory, \cite{Adams}.

For $s\geq 0$, let $X^{s}$ be the product space of $d$ copies of $H^{s}$. The inner product in $X^{0}$ is denoted by
\begin{eqnarray*}
\langle \varphi,\psi\rangle=\sum_{j=1}^{d}(\varphi_{j},\psi_{j})_{0},\quad \varphi=(\varphi_{1},\ldots,\varphi_{d})^{T}, \psi=(\psi_{1},\ldots,\psi_{d})^{T}\in X^{0},
\end{eqnarray*}
with associated norm 
\begin{eqnarray*}
||\varphi||_{0}=\left(\sum_{j=1}^{d}|\varphi_{j}|_{0}^{2}\right)^{1/2}.
\end{eqnarray*}
This is used to define the norm in $X^{k}, k\geq 0$ integer, given by
\begin{eqnarray*}
||\varphi ||_{k}^{2}=\sum_{j=0}^{k}\left\|\frac{d^{j}}{dx^{j}}\varphi\right\|_{0}^{2},\quad \varphi=(\varphi_{1},\ldots,\varphi_{d})^{T}\in X^{k},
\end{eqnarray*}
where $\frac{d^{j}}{dx^{j}}\varphi=(\frac{d^{j}}{dx^{j}}\varphi_{1},\ldots,\frac{d^{j}}{dx^{j}}\varphi_{d})^{T}, 0\leq j\leq k$.
Similarly, we will consider the product spaces $X_{0}^{s}$ of $d$ copies of $H_{0}^{s}, s\geq 0$.

On the other hand, the dual space of $H^{k}(\Omega)$ will be denoted by $\left(H^{k}(\Omega)\right)^{\prime}$; this is defined from the completion of $L^{2}(\Omega)$ with respect to the norm, \cite{Adams},
\begin{eqnarray*}
||v||_{-k,2}=\sup_{u\in H^{k}(\Omega),||u||_{k}=1}|\langle u,v\rangle |,\quad
\langle u,v\rangle=\int_{\Omega}uvd\Omega.
\end{eqnarray*}
Additionally, $(X^{k})'$ will stand for $(H^{k}(\Omega))' \times\ldots^{d)}\times (H^{k}(\Omega))'$.

For $T>0$, $\Omega_{T}=\Omega\times (0,T]$ will denote the set of points $({ x},t), { x}\in\Omega, 0<t\leq T$ and $\overline{\omega_{T}}:=\overline{\Omega}\times [0,T]$. The space of infinitely continuously differentiable {real-valued} functions in $\overline{\Omega}\times (0,T]$ will be denoted by $C^{\infty}\left(\overline{\Omega}\times (0,T]\right)$ as well as the space of $m-$th order continuously differentiable functions ${\bf u}:(0,T]\rightarrow X^{k}$ by $C^{m}(0,T,X^{k})$, $m,k$ nonnegative integers. Additionally, $L^{2}(0,T,X^{k})$  will stand for the normed space of functions $u:(0,T]\rightarrow X^{k}(\Omega)$ with associated norm
\begin{eqnarray*}
||u||_{L^{2}(0,T,X^{k})}=\left(\int_{0}^{T}||u(t)||_{k}^{2}dt\right)^{1/2}.
\end{eqnarray*}
We also denote by $L^{\infty}(0,T,X^{k})$ the normed space of functions $u:(0,T]\rightarrow X^{k}(\Omega)$ with norm
\begin{eqnarray*}
||u||_{L^{\infty}(0,T,X^{k})}=
{\esssup}_{t\in (0,T)}||u(t)||_{k},
\end{eqnarray*}
with $\esssup$ as the essential supremum.

Throughout the paper $C$ will be used to denote a generic, 
positive constant.

\section{Mathematical properties of the ibvp \eqref{eq:3psystem1a}-\eqref{eq:3psystem1c}}
\label{sec2}
This section is concerned with the weak formulation of \eqref{eq:3psystem1a}-\eqref{eq:3psystem1c} and its well-posedness, under (H1)-(H3) and the simplifications (i)-(iii).

\subsection{Weak formulation}
\label{sec22}
Taking the $L^{2}$-inner product of  \eqref{eq:3psystem1a} with $v\in X_{0}^{1}$, integration by parts and the assumption (ii) lead to the
weak form of (\ref{eq:3psystem1a})-(\ref{eq:3psystem1c}): find $u:[0,T]\rightarrow X_{0}^{1}$ such that
\begin{eqnarray}
\mathcal{A}_{u}(u_{t},v)&=&\mathcal{B}_{u}(u,v),\quad v\in X_{0}^{1},\label{eq:21}\\
u(0)&=&u_{0},\nonumber
\end{eqnarray}
where 
\begin{eqnarray}
\mathcal{A}_{u}(\varphi,\psi)&=&\langle \varphi,\psi\rangle+L_{A(u)}(\varphi,\psi),\nonumber\\
\mathcal{B}_{u}(\varphi,\psi)&=&L_{B(u)}(\varphi,\psi)+\langle \frac{\partial}{\partial x}G(\varphi),\psi\rangle+\langle {\gamma}(\varphi),\psi\rangle,\quad\varphi,\psi\in X_{0}^{1},\label{eq:22}
\end{eqnarray}
and where if $C=(c_{ij}(u))_{i,j=1}^{d}\in\mathbb{R}^{d\times d}$
\begin{eqnarray}
L_{C(u)}(\varphi,\psi)=\int_{\Omega}(C(u)\varphi_{x})\cdot \psi_{x}dx.\label{eq:23}
\end{eqnarray}
The dot in (\ref{eq:23}) stands for the Euclidean inner product in $\mathbb{R}^{d}$.

\begin{remark}
\label{remark1}
Some properties of the formulation (\ref{eq:21}) will be used in the sequel. We recall that for the scalar case, the bilinear form
\begin{eqnarray*}
l(\varphi,\psi)=\int_{\Omega}\varphi_{x}\psi_{x}dx,\; \varphi,\psi\in H^{1},\label{bil1}
\end{eqnarray*}
(which is a seminorm in $H^{1}$) is continuous in $H^{1}\times H_{0}^{1}$ and elliptic in $H_{0}^{1}\times H_{0}^{1}$, \cite{BernardiM1997,CanutoHQZ1988}. A direct extension of this result is that the bilinear form
\begin{eqnarray}
L(\varphi,\psi)=\int_{\Omega}\varphi_{x}\cdot\psi_{x}dx,\quad \varphi,\psi\in X^{1},\label{bil2}
\end{eqnarray}
is continuous in $X^{1}\times X_{0}^{1}$ and elliptic in $X_{0}^{1}\times X_{0}^{1}$. Then writing
$$\langle \frac{\partial}{\partial x}G(\varphi),\psi\rangle=-\int_{\Omega}G(\varphi)\cdot \psi_{x}dx,\; \varphi\in X^{1},\; \psi\in X_{0}^{1},$$ and the hypothesis (H2) imply that the right-hand side of (\ref{eq:21}) is continuous in $X^{1}\times X_{0}^{1}$.

On the other hand, let $v\in H_{0}^{1}$ be fixed and consider the functional
\begin{eqnarray}
L_{C(v)}(\varphi,\psi)=\int_{\Omega}C(v)\varphi_{x}\cdot\psi_{x}dx,\quad \varphi,\psi\in X^{1},\label{bil3}
\end{eqnarray}
where $C=C(v)=(c_{ij}(v))_{i,j=1}^{d}$ denotes a $d\times d$ matrix. A comparison with (\ref{bil2}) directly implies that:
\begin{itemize}
\item If $c_{ij}(v), 1\leq i,j\leq d,$ is bounded, then (\ref{bil3}) is continuous in $X^{1}\times X_{0}^{1}$: there is a constant $C_{1}>0$ such that 
\begin{eqnarray*}
|L_{C(v)}(\varphi,\psi)|\leq C_{1}||\varphi||_{1}||\psi||_{1},\quad \varphi,\psi\in X^{1}. \label{bil31}
\end{eqnarray*}
\item If $C$ is uniformly positive definite, cf. (H1), then 
 there is a constant $C_{2}>0$ such that 
\begin{eqnarray}
|L_{C(v)}(\varphi,\varphi)|\geq C_{2}||\varphi||_{1}^{2},\quad \varphi,\in X_{0}^{1}, \label{bil32}
\end{eqnarray}
where in (\ref{bil32})  the equivalence in $H_{0}^{1}$ between the norm $||\cdot||_{1}$ and the seminorm (\ref{bil1}) is used.
\end{itemize}
\end{remark}

\subsection{Well-posedness}
\label{sec23}
Some well-posedness results are now analyzed. We first define
$$W(0,T)=\{w\in L^{2}(0,T,X^{1}), \frac{dw}{dt}\in L^{2}(0,T,(X^{1})^{\prime})\},$$ provided with the graph norm, \cite{Brezis}. In order to study the existence of solutions of (\ref{eq:21}), we first consider, for some given $U\in W(0,T)\cap L^{\infty}(0,T,X_{1})$, the following ibvp in $\overline{\Omega}_{T}$
\begin{eqnarray*}
&&\left(I-\frac{\partial}{\partial x}\left(A(U)\frac{\partial}{\partial x}\right)\right)\frac{\partial}{\partial t}u
=-\frac{\partial}{\partial x}\left(B(u)\frac{\partial}{\partial x}u\right)+\frac{\partial}{\partial x}G(u)+\gamma(u),\label{eq:l3psystem1a}\\
&&u(\pm 1,t)=0,\quad0\leq t\leq T,\label{eq:l3psystem1b}\\
&&u(x,0)=u_{0}(x),\quad -1\leq x\leq 1,\label{eq:l3psystem1c}
\end{eqnarray*}
which can be written in a weak form as: find $u:[0,T]\rightarrow X_{0}^{1}$ such that, for all $\psi\in X_{0}^{1}$
\begin{eqnarray}
\mathcal{A}_{U}(u_{t},v)&=&\mathcal{B}_{u}(u,v),\quad v\in X_{0}^{1},\label{eq:21lin}\\
u(0)&=&u_{0},\nonumber
\end{eqnarray}
where $\mathcal{A}, \mathcal{B}$ are defined in (\ref{eq:22}). The following lemma ensures the well-posedness of (\ref{eq:21lin}).
\begin{lemma}
\label{lemma1}
Let $T>0$, $U\in W(0,T)\cap L^{\infty}(0,T,X_{1})$, and assume that (H1), (H2) hold. If $u_{0}\in X_{0}^{1}$, then there is a unique solution $u\in C(0,T,X_{0}^{1})$ of (\ref{eq:21lin}). Furthermore, there is a constant $C$ depending on $||u_{0}||_{1,w}, A, B, G, \gamma$ such that
\begin{eqnarray}
||u||_{L^{\infty}(0,T,X^{1})}+||u_{t}||_{L^{\infty}(0,T,X^{1})}\leq C.\label{5b}
\end{eqnarray}
\end{lemma}
\begin{proof}
The proof will follow a variant of the method of Faedo-Galerkin, \cite{Lions1969}. Let $\{W_{k}\}_{k}$ be an orthonormal basis of $X_{0}^{1}$ and $V_{n}={\rm span}(W_{1},\ldots,W_{n})$. We consider the approximate problem of determining $u_{n}\in H^{1}(0,T,V_{n})$ satisfying
\begin{eqnarray}
\mathcal{A}_{U}(\partial_{t}u_{n},W_{s})&=&\mathcal{B}_{u_{n}}(u_{n},W_{s}),\quad s=1,\ldots,n,\nonumber\\
u_{n}(0)&=&u_{0}^{(n)},\label{eq:212}
\end{eqnarray}
where 
\begin{eqnarray*}
u_{0}=\sum_{k=1}^{\infty} u_{0,k}W_{k},\quad u_{0,k}=\langle u_{0},W_{k}\rangle,\quad k\geq 1,\quad
u_{0}^{(n)}= \sum_{k=1}^{n} u_{0,k}W_{k}.
\end{eqnarray*}
Writing
$$u_{n}=\sum_{k=1}^{n}C_{k}(t)W_{k},$$ then (\ref{eq:212}) leads to an equivalent ode system
\begin{eqnarray*}
C_{s}'(t)+\sum_{k=1}^{n}L_{A(U)}(W_{k},W_{s})C_{k}'(t)&=&\sum_{k=1}^{n}L_{B(u_{n})}(W_{k},W_{s})C_{k}(t)\nonumber\\
&&+\sum_{k=1}^{n}\langle G'(u_{n})\partial_{x}W_{k},W_{s}\rangle C_{k}(t)+\langle\gamma(u_{n}),W_{s}\rangle,\label{eq:213}
\end{eqnarray*}
with $C_{s}(0)=u_{0}^{(s)}, s=1,\ldots,n$. Using (H1), (H2), the standard ode systems theory proves the local existence of $u_{n}\in C^{1}(0,t^{n},V_{n})$ for some $t^{n}\in (0,T]$.

If we take $W_{s}=\partial_{t}u_{n}$ in (\ref{eq:212}), using (H1), (H2), and Remark \ref{remark1}, we have
\begin{eqnarray}
||\partial_{t}u_{n}(s)||_{1}\leq C ||u_{n}(s)||_{1},\quad s\in (0,t^{n}),\label{eq:213b}
\end{eqnarray}
for some constant $C$ independent of $n$. Then, for $t\in (0,t^{n})$
\begin{eqnarray}
||u_{n}(t)||_{1}=\left\|u_{n}(0)+\int_{0}^{t}\partial_{t}u_{n}(s)ds\right\|_{1}\leq  ||u_{n}(0)||_{1,w}+C\int_{0}^{t}||u_{n}(s)||_{1}ds.\label{eq:214}
\end{eqnarray}
Using the orthogonality of the $W_{k}$, there is a constant $C$ such that
\begin{eqnarray}
||u_{n}(0)||_{1,w}\leq C||u_{0}||_{1,w}.\label{eq:215}
\end{eqnarray}
Therefore, (\ref{eq:214}), (\ref{eq:215}), and Gronwall's lemma lead to
\begin{eqnarray}
||u_{n}(t)||_{1,w}\leq C,\quad t\in (0,t^{n}),\label{eq:216}
\end{eqnarray}
with $C$ independent of $n$. It follows that we can take $t^{n}=T$ for all $n$. Using again (\ref{eq:213b}) and (\ref{eq:216}) for $t\in [0,T]$ we have
\begin{eqnarray*}
||\partial_{t}u_{n}||_{L^{\infty}(0,T,X^{1})}\leq C.
\end{eqnarray*}
Therefore, for $t\in [0,T]$, the sequences
$$\{u_{n}(t)\}_{n}, \quad \{\partial_{t}u_{n}(t)\}_{n},$$ are uniformly bounded in $H^{1}$. This implies the existence of subsequences (denoted in the same way for simplicity) and $\varphi(t),\psi(t)\in H^{1}$ such that
$$u_{n}(t)\rightarrow \varphi(t),\quad \partial_{t}u_{n}(t)\rightarrow \psi(t),$$ weakly in $X^{1}$ and for $0\leq t\leq T$. On the other hand, due to the compact embedding of $H^{1}$ into $L^{2}$, there is $u\in L^{2}(0,T,X^{0})$ such that
$$||u_{n}-u||_{L^{2}(0,T,X^{0})}\rightarrow 0,\; n\rightarrow\infty,$$
for some subsequence $u_{n}$. The subsequence can also be chosen to converge almost everywhere on $\Omega_{T}$, \cite{Brezis}. It transpires that $u=\varphi\in L^{\infty}(0,T,X^{1})$ and $\psi=u_{t}$. From (H1), (H2), and Remark \ref{remark1}, taking the limit in (\ref{eq:212}) implies that $u$ satisfies (\ref{eq:21lin}) for any $W_{s}$ and therefore for any $v\in X_{0}^{1}$. The estimate (\ref{5b}) comes from evaluating (\ref{eq:21lin}) at $v=u_{t}$ and using the arguments above that lead to (\ref{eq:213b}) and (\ref{eq:216}), for $u$ instead of $u_{n}$ and all $t\in (0,T)$.

As for uniqueness, let $u_{1}, u_{2}$ be solutions of  (\ref{eq:21lin}) with the same initial condition and $u=u_{1}-u_{2}$. Then, for $\psi\in X_{0}^{1}$
\begin{eqnarray}
\langle u_{t},\psi\rangle+L_{A(U)}(u_{u},\psi)&=&\underbrace{\int_{\Omega}\left(B(u_{1})u_{1x}-B(u_{2})u_{2x}\right))\cdot \psi dx}_{I_{1}}\nonumber\\
&&+\underbrace{\langle \partial_{x}(G(u_{1})-G(u_{2})),\psi\rangle}_{I_{2}}\nonumber\\
&&+\underbrace{\langle \gamma(u_{1})-\gamma(u_{2}),\psi\rangle}_{I_{3}}.\label{*6}
\end{eqnarray}
If we write
$$B(u_{1})u_{1x}-B(u_{2})u_{2x}=B(u_{1})(u_{1x}-u_{2x})+(B(u_{1})-B(u_{2}))u_{2x},$$ and use (H1), (H2), and Remark \ref{remark1}, then we can obtain the existence of a constant $C>0$ such that
\begin{eqnarray}
|I_{1}|+|I_{2}|+|I_{3}|\leq C||u_{1}-u_{2}||_{1}||\psi||_{1}.\label{**6}
\end{eqnarray}
Therefore, evaluating (\ref{*6}) at $\psi=u_{t}$, using Remark \ref{remark1} and (\ref{**6}), it holds that for all $t\in [0,T]$
$$||u_{t}(t)||_{1}^{2}\leq C ||u(t)||_{1}||u_{t}(t)||_{1}.$$ Thus
$$||u(t)||_{1}\leq ||u(0)||_{1}+C\int_{0}^{t}||u(s)||_{1}ds,$$ and Gronwall's lemma implies that $u(t)=0, t\in [0,T]$.
\end{proof}
Lemma \ref{lemma1} is used to prove the main theorem of well-posedness.
\begin{theorem}
\label{theorem21} Let $T>0$ and assume that, under simplifications (i)-(iii), the problem
(\ref{eq:3psystem1a})-(\ref{eq:3psystem1c}) satisfies (H1), (H2).
If $u_{0}\in X_{0}^{1}$, then there is a unique 
solution $u\in C(0,T,X_{0}^{1})$ of (\ref{eq:21}) that depends continuously on the initial data. Furthermore, there is a constant $C$ depending on $||u_{0}||_{1}, A, B, G, \gamma$ such that
\begin{eqnarray}
||u||_{L^{\infty}(0,T,X^{1})}+||u_{t}||_{L^{\infty}(0,T,X^{1})}\leq C.\label{bound1}
\end{eqnarray}
In addition, if $k>1$ is an integer, $m=\max\{1,k-1\}$, assume that $a_{ij}\in C^{k}$, $B_{ij}, G_{i},\gamma_{i}\in C^{m}, i,j=1,\ldots d$, and $u_{0}\in X_{0}^{k}$. Let $u:[0,T]\rightarrow  X_{0}^{1}$ be satisfying (\ref{eq:21}). Then $u(t)\in X_{0}^{k}$ for all $t\in [0,T]$ and there is a constant  $C$ depending on $||u_{0}||_{k}, A, B, G, \gamma$ such that
\begin{eqnarray}
||u||_{L^{\infty}(0,T,X^{k})}+||u_{t}||_{L^{\infty}(0,T,X^{k})}\leq C.\label{5bb}
\end{eqnarray}
\end{theorem}
\begin{proof}
The existence of solution is derived from the estimate (\ref{5b}) in Lemma \ref{lemma1} and the application of the Schrauder fixed-point theorem, \cite{Brezis}. We consider the nonempty, convex subset of $W(0,T)$
\begin{eqnarray*}
K=\{{w}\in W(0,T): {w} \,\, \mbox{satisfies}\,\, \mbox{(\ref{5b})}\,\,\mbox{with}\,\, {w}(0)={u}_{0}\},
\end{eqnarray*}
and the mapping $T:K\longrightarrow W(0,T)$ such that $T({w}):={u}({w})$ is the  solution of (\ref{eq:21lin}) with $U={w}$.

Note first that, by construction $T(K)\subset K$. On the other hand, let $\{w_{n}\}$ be a sequence in $K$ and $t\in [0,T]$. Since $K$ is bounded, then
$$\{w_{n}(t)\}_{n}, \quad \{\partial_{t}w_{n}(t)\}_{n},$$ are uniformly bounded in $H^{1}$. The same arguments used in the proof of Lemma \ref{lemma1} can be applied here to obtain the existence of a subsequence (denoted again by $\{w_{n}\}_{n}$) and $w\in L^{2}(0,T,X^{1})$ such that
\begin{itemize}
\item $w_{n}(t)\rightarrow w(t),\quad \partial_{t}w_{n}(t)\rightarrow w_{t}(t),$ weakly in $X^{1}$ and $(X^{1})^{\prime}$ respectively, and for $0\leq t\leq T$.
\item $||w_{n}-w||_{L^{2}(0,T,X^{0})}\rightarrow 0,\; n\rightarrow\infty,$ and the convergence is almost everywhere on $\Omega_{T}$.
\end{itemize}
In particular, $K$ is weakly compact in $W(0,T)$. Finally,  let $\{w_{n}\}_{n}$ be a sequence in $K$ with $w_{n}\rightarrow w$ weakly for some $w\in K$, and let $u_{n}=T(w_{n})$. The same arguments as those above, applied to $u_{n}$, and the property that $K$ is weakly compact in $W(0,T)$ imply the existence of a subsequence (denoted again by $u_{n}=T(w_{n})$) and $u\in L^{2}(0,T,X^{1})$ such that
 \begin{itemize}
\item[(i)] $u_{n}\rightarrow u,\; \partial_{t}u_{n}\rightarrow u_{t},$ weakly in $L^{2}(0,T,X^{1})$ and $L^{2}(0,T,(X^{1})^{\prime})$ respectively.
\item[(ii)] $u_{n}\rightarrow u$ in ${L^{2}(0,T,X^{0})}$ and almost everywhere on $\Omega_{T}$.
\end{itemize}
These and the properties of (\ref{bil1}) also imply that $u_{n}(0)\rightarrow u(0)$ in $(X^{1})^{\prime}$ and that $\partial_{x}u_{n}\rightarrow \partial_{x}u$ weakly in ${L^{2}(0,T,X^{0})}$. In addition, hypotheses (H1), (H2), Remark \ref{remark1}, and property (iii) above imply that
$$A(w_{n})\rightarrow A(w), \quad B(w_{n})\rightarrow B(w), \quad G(w_{n})\rightarrow G(w), \quad \gamma(w_{n})\rightarrow \gamma(w),$$ in ${L^{2}(0,T,X^{0})}$. Therefore, if we take the limit in (\ref{eq:21lin}) then $u=T(w)$. On the other hand, since the whole sequence $\{u_{n}\}_{n}$ is bounded in $K$ which is weakly compact, then it converges weakly in $W(0,T)$. By uniqueness of solution of the problem (\ref{eq:21lin}), the weak limit must be $u=T(w)$. All this proves that $T$ is weakly continuous and therefore the Schrauder fixed-point theorem proves the existence of a solution $u$ of (\ref{eq:21}) which is in $K$ and therefore
$$u\in L^{2}(0,T,X^{1}), \quad \partial_{t}u\in L^{2}(0,T,(X^{1})^{\prime}),$$ and satisfying (\ref{5b}). Furthermore, (\ref{5b}) and (H1), (H2) imply that $u\in C(0,T,X^{1})$.

As for uniqueness, let $u_{1}, u_{2}$ be solutions of  (\ref{eq:21}) with the same initial condition and $u=u_{1}-u_{2}$. Then, for $\psi\in X_{0}^{1}$
\begin{eqnarray}
\langle u_{t},\psi\rangle+\int_{\Omega}\left(A(u_{1})\partial_{t}u_{1x}-A(u_{2})\partial_{t}u_{2x}\right))\cdot \psi_{x} dx&=&\nonumber\\
\underbrace{\int_{\Omega}\left(B(u_{1})u_{1x}-B(u_{2})u_{2x}\right))\cdot \psi dx}_{I_{1}}&&\nonumber\\
+\underbrace{\langle \partial_{x}(G(u_{1})-G(u_{2})),\psi\rangle}_{I_{2}}+\underbrace{\langle \gamma(u_{1})-\gamma(u_{2}),\psi\rangle}_{I_{3}}.&&\label{*7}
\end{eqnarray}
On the other hand, since
$$A(u_{1})\partial_{t}u_{1x}-A(u_{2})\partial_{t}u_{2x}=A(u_{1})\partial_{t}u_{x}+(A(u_{1})-A(u_{2}))\partial_{t}u_{2x},$$ then (\ref{*7}) can be rewritten as
\begin{eqnarray}
\langle u_{t},\psi\rangle+\int_{\Omega}\left(A(u_{1})\partial_{t}u_{x}\right))\cdot \psi_{x} dx
&=&\underbrace{\int_{\Omega}\left((A(u_{2})-A(u_{1}))\partial_{t}u_{2x}\right))\cdot \psi_{x} dx}_{I_{0}}\nonumber\\
&&+
\underbrace{\int_{\Omega}\left(B(u_{1})u_{1x}-B(u_{2})u_{2x}\right))\cdot \psi dx}_{I_{1}}\nonumber\\
&&+\underbrace{\langle \partial_{x}(G(u_{1})-G(u_{2})),\psi\rangle}_{I_{2}}\nonumber\\
&&+\underbrace{\langle \gamma(u_{1})-\gamma(u_{2}),\psi\rangle}_{I_{3}}.\label{*7b}
\end{eqnarray}
Note now that, by using (H1), (H2), (\ref{5b}), and Remark \ref{remark1}, the right-hand side of (\ref{*7b}) can be bounded in a similar way to that of (\ref{*6}). Furthermore, evaluating (\ref{*7b}) at $\psi=u_{t}$ and  because of (H1), the left-hand side is bounded from below leading to
$$||u_{t}(t)||_{1}^{2}\leq C ||u(t)||_{1}||u_{t}(t)||_{1},$$ for some constant $C$. Hence
\begin{eqnarray}
||u(t)||_{1}\leq ||u(0)||_{1}+C\int_{0}^{t}||u(s)||_{1}ds,\label{*8}
\end{eqnarray}
and Gronwall's lemma yields $u(t)=0, t\in [0,T]$.

We make use of (\ref{*8}), which is valid for any solutions $u_{1}, u_{2}$ of the first equation in (\ref{eq:21}), to prove the continuous dependence on the initial data. Using Gronwall's lemma, if $u=u_{1}-u_{2}$, from (\ref{*8}) we have
$$||u(t)||_{1}\leq ||u(0)||_{1}e^{tC},\, 0\leq t\leq T.$$ Therefore, if $\epsilon>0$ and $\delta=\epsilon e^{-TC}$, then

$$||u_{1}(0)-u_{2}(0)||_{1}<\delta\quad {\rm implies}\quad ||u_{1}(t)-u_{2}(t)||_{1}<\epsilon,\; 0\leq t\leq T.$$

Finally, the regularity of the solution can be derived from a classical bootstrapping argument and the additional hypotheses on the coefficients of (\ref{eq:21}).
\end{proof}

\section{Numerical approximation}
\label{sec3}
In this section the numerical approach, based on Legendre-Galerkin spectral semidiscretization in space and a temporal discretization with SSP methods, is introduced and analyzed. 
\subsection{The family of Legendre polynomials. Projection errors}
Collected here are some results on polynomial approximation that will be used below, cf. e.~g. \cite{BernardiM1997,CanutoHQZ1988} for details.
For an integer $N\geq 2$, $\mathbb{P}_{N}=\mathbb{P}_{N}(\Omega)$ will 
stand for the space of polynomials on $\Omega$ of degree at 
most $N$ and
\begin{eqnarray*}
\mathbb{P}_{N}^{0}=\mathbb{P}_{N}^{0}(\Omega):=\mathbb{P}_{N}(\Omega)\cap H_{0}^{1}(\Omega),
\end{eqnarray*}
is the subspace of polynomials in $\mathbb{P}_{N}(\Omega)$ vanishing on the boundary $\partial\Omega$. As usual $\mathbb{P}_{N}^{d}$ (resp. $(\mathbb{P}_{N}^{0})^{d}$) will denote the product of $d$ copies of $\mathbb{P}_{N}$ (resp. $\mathbb{P}_{N}^{0}$).

The $k$th Legendre polynomial $L_{k}$ is defined as the eigenfunction of the singular Sturm-Liouville problem on $(-1,1)$
$$-(pu')'+qu=\lambda wu,$$ with $p(x)=1-x^{2}, q(x)=0, w(x)=1$ and eigenvalue $\lambda=-k(k+1), k=0,1,\ldots$. The family $\{L_{k}\}_{k\geq 0}$ is orthogonal with respect to the weight function $w(x)$, and it is a subfamily of the Jacobi family of orthogonal polynomials with respect to the weight $(1-x^{2})^{\mu}, -1<\mu<1$. (The Legendre case corresponds to $\mu=0$.) For $u\in L^{2}(\Omega)$, the polynomial
$$P_{N}(x)=\sum_{k=0}^{N}\widehat{u}_{k}L_{k}(x),\; \widehat{u}_{k}=\frac{(u,L_{k})_{0}}{|L_{k}|_{2}^{2}},$$ is the orthogonal projection of $u$ on $\mathbb{P}_{N}$, that is
$$(P_{N}u,p)=(u,p),\quad p\in \mathbb{P}_{N}.$$
We also denote by $P_{N}^{10}v\in \mathbb{P}_{N}^{0}$ 
 the orthogonal projection of $v$ with respect 
to the inner product in $H_{0}^{1}$ given by (\ref{bil1}).

Some estimates for the projection errors will be necessary in the analysis below. They are, \cite{BernardiM1989,BernardiM1997}
\begin{eqnarray}
|v-P_{N}v|_{0}&\leq &C N^{-s}|v|_{s},\quad v\in H^{s},\quad s\geq 0,\label{c42a}\\
|v-P_{N}v|_{r}&\leq &C N^{r-s}|v|_{s},\quad v\in H^{s},\quad 1\leq r\leq s,\quad r\; {\rm integer},\nonumber
\end{eqnarray}
and for $v\in H^{s}\cap H_{0}^{1}$
\begin{eqnarray}
|v-P_{N}^{10}v|_{1}\leq C N^{1-s}|v|_{s},\quad
s\geq 1.\label{c42c}
\end{eqnarray}
The previous definitions and the estimates (\ref{c42a}), (\ref{c42c}) will be used componentwise when dealing with $v=(v_{1},\ldots,v_{d})^{T}\in X^{s}$ or $X_{0}^{s}, s\geq 0$, and by abuse of notation we will denote by $P_{N}v$ and $P_{M}^{10}v$ the vectors
$$P_{N}v=(P_{N}v_{1},\ldots,P_{N}v_{d})^{T},\quad P_{N}^{10}v=(P_{N}^{10}v_{1},\ldots,P_{N}^{10}v_{d})^{T},$$ respectively.

The projection operator $P_{N}^{10}$, corresponding to the bilinear form (\ref{bil1}), can be extended to functionals of the form
\begin{eqnarray}
l_{a(U)}(\varphi,\psi)=\int_{\Omega}a(U)\varphi_{x}\psi_{x}dx,\quad \varphi,\psi\in H^{1},\label{bil4}
\end{eqnarray}
for a fixed $U\in H_{0}^{1}$, $a=a(U)$ some Lipschitz continuous function bounded above and below by positive constants, as well as the vector version
\begin{eqnarray}
L_{C(U)}(\varphi,\psi)=\int_{\Omega}C(U)\varphi_{x}\cdot\psi_{x}dx,\quad \varphi,\psi\in X^{1},\label{bil5}
\end{eqnarray}
where $C=C(U)=(c_{ij}(U))_{i,j=1}^{d}$ satisfies (H1), with $c_{ij}(U), i,j=1,\ldots,d$, bounded. Thus, if
$v\in H_{0}^{1}$, then the orthogonal projection 
$\overline{v}\in \mathbb{P}_{N}^{0}$ of $v$ with 
respect to the bilinear form involving (\ref{bil4}) 
\begin{eqnarray*}
{a}_{U}(\varphi,\psi)=(\varphi,\psi)_{0}+l_{a(U)}(\varphi,\psi),\quad \varphi,\psi\in H^{1},
\end{eqnarray*}
is defined as $\overline{v}=R_{N}v\in 
\mathbb{P}_{N}^{0}$ such that
\begin{eqnarray*}
{a}_{U}(\overline{v}-v,\psi)=0,\quad \psi\in \mathbb{P}_{N}^{0}.\label{ad27}
\end{eqnarray*}
For this projection, following \cite{BernardiM1989}, it holds that
\begin{eqnarray}
|v-\overline{v}|_{1}+N|v-\overline{v}|_{0}\leq C N^{1-m}|v|_{m},\label{ad27b}
\end{eqnarray}
for $v\in H^{m}\cap H_{0}^{1}, m\geq 1$. Furthermore, a similar argument to that exposed in 
\cite{EAAD20} shows that
\begin{eqnarray}
|v-\overline{v}|_{2}
\leq  C N^{3-m}|v|_{m}.\label{ad27d}
\end{eqnarray}
Similarly, the corresponding orthogonal projection 
$$\overline{v}\in (\mathbb{P}_{N}^{0})^{d}$$ of $v=(v_{1},\ldots,v_{d})^{T}\in X^{1}$ with 
respect to the functional defined from (\ref{bil5}) 
\begin{eqnarray*}
\mathcal{C}_{U}(\varphi,\psi)=\langle\varphi,\psi\rangle+L_{C(U)}(\varphi,\psi)\; \varphi,\psi\in X^{1},\label{ad27e}
\end{eqnarray*}
is given by $\overline{v}=(\overline{v}_{1},\ldots,\overline{v}_{d})^{T}, $ such that
 \begin{eqnarray}
\mathcal{C}_{U}(\overline{v}-v,\psi)=0,\; \psi\in (\mathbb{P}_{N}^{0})^{d}.\label{ad27ee}
\end{eqnarray}
The corresponding estimates in the norm $||\cdot ||_{k}, k=0,1,2$, are derived by components from (\ref{ad27b}), (\ref{ad27d}).

\subsection{Legendre-Galerkin spectral semidiscretization}
\label{sec31}
Let $N\geq 2$ be an integer,  and $u_{0}\in X_{0}^{1}$. 
The semidiscrete Galerkin approximation to (\ref{eq:21}) is defined as 
the function $u^{N}:[0,T]\rightarrow (\mathbb{P}_{N}^{0})^{d}$ satisfying
\begin{eqnarray}
\mathcal{A}_{u^{N}}(u_{t}^{N},\psi)=\mathcal{B}_{u^{N}}(u^{N},\psi),\quad \psi\in (\mathbb{P}_{N}^{0})^{d},\label{ad31a}
\end{eqnarray}
where $\mathcal{A}, \mathcal{B}$ are defined in (\ref{eq:22}), and with 
\begin{eqnarray}
\mathcal{A}_{u^{N}}(u_{N}(0),\psi)=\mathcal{A}_{u^{N}}(u_{0},\psi),\quad \psi\in (\mathbb{P}_{N}^{0})^{d}.\label{ad31b}
\end{eqnarray}

Before analyzing the existence of the semidiscrete solution and its convergence to the exact solution of (\ref{eq:21}), it may be worth mentioning some details of the implementation. This is typically made via the so-called Galerkin-Numerical Integration (G-NI) formulation,  which is obtained from the use of quadrature formulas of Gauss type and different families of weights and nodes. Considered here is the following G-NI formulation of the Legendre-Galerkin method, based on the representation in the nodal basis of $\mathbb{P}_{N}^{0}$
%
%
%
%
\begin{eqnarray}
\psi_{j}(x)=\frac{1}{N(N+1)}\frac{(1-x^{2})}{(x_{j}-x)}\frac{L_{N}^{\prime}(x)}{L_{N}(x_{j})},\;j=0,\ldots,N,\label{nodalb}
\end{eqnarray}
where $x_{j}, j=0,\ldots,N$, denotes the nodes associated 
to the Legendre-Gauss-Lobatto quadrature, \cite{BernardiM1997,CanutoHQZ1988}, $L_{N}$ is 
the $N$-th Legendre polynomial. 
Since (\ref{nodalb}) satisfies
\begin{eqnarray}
\psi_{j}(x_{k})=\delta_{jk},\; j,k=0,\ldots,N,\label{nodalb2}
\end{eqnarray}
and $x_{0}=-1, x_{N}=1$, then
we write
\begin{eqnarray}
&&u^{N}(x,t)=(u_{1}^{N}(x,t),\ldots,u_{d}^{N}(x,t))^{T},\nonumber\\
&&u_{j}^{N}(x,t)=\sum_{k=1}^{N-1}U_{jk}(t)\psi_{k}(x),\; j=1,\ldots,d,\label{Leg_Gal}
\end{eqnarray}
(this includes directly the boundary conditions into the representation) and it is clear from (\ref{nodalb2}), (\ref{Leg_Gal}) that
\begin{eqnarray*}
U_{jk}(t)=u_{j}^{N}(x_{k},t),\quad k=1,\ldots,N-1, \quad j=1,\ldots,d.
\end{eqnarray*}
By inserting (\ref{Leg_Gal}) into (\ref{ad31a}) evaluated at
$$\psi=(\psi_{j},0,\ldots,0)^{T},\ldots, (0,\ldots,0,\psi_{j})^{T},\; j=1,\ldots,N-1,$$ the  G-NI formulation consists of approximating the resulting integrals by the Legendre-Gauss-Lobatto quadrature leading to a $d(N-1)$ ode system for
$U=(U_{jk}(t)), 1\leq j\leq d, 1\leq k\leq N-1$, of the form, cf. \cite{EAAD20}
\begin{eqnarray}
K(u^{N})\frac{d}{dt}U=L(u^{N})U+\mathcal{H}(u^{N}),\label{semidL1}
\end{eqnarray}
where 
\begin{eqnarray}
&&K(u^{N})=\left(K^{(0)}+K^{(2)}(A)\right),\; L(u^{N})=K^{(2)}(B)+K^{(1)}(G),\nonumber\\
&&K^{(0)}={\rm diag}\left(K_{N}^{(0)},\ldots^{d)},K_{N}^{(0)}\right),\quad 
K_{N}^{(0)}={\rm diag}(w_{1},\ldots,w_{N-1}),\label{GNI1}\\
&&K^{(2)}(C)=\left(K_{pq}^{(2)}(C)\right)_{p,q=1}^{d},\nonumber\\ 
&&\left(K_{pq}^{(2)}(d)\right)_{jk}=\sum_{h=0}^{N}d_{pq}(U_{h})
\frac{d\psi_{j}}{dx}(x_{h})\frac{d\psi_{k}}{dx}(x_{h})w_{h},\nonumber\\
&& j,k=1,\ldots, N-1,\; p,q=1,2,\cdots,d,\; C=A\;{\rm or}\; B,\label{GNI2}\\
&&K^{(1)}(G)=
\left(K_{pq}^{(1)}(G)\right)_{p.q=1}^{d}\quad \left(K_{pq}^{(1)}(G)\right)_{jk}=g_{pq}(U_{j})
\frac{d\psi_{k}}{dx}(x_{j})w_{j},\nonumber\\
&& j,k=1,\ldots,N-1,\; g=G',\label{GNI3}\\
&&\mathcal{H}(u^{N})={\rm diag}\left(\Gamma_{N}^{(1)},\ldots, \Gamma_{N}^{(d)}\right),\quad (\Gamma_{N}^{(p)})_{j}=w_{j}\gamma_{p}(U_{j}), \nonumber\\
&& p=1,\ldots,d,\; j=1,\ldots,N-1,\label{GNI4}
\end{eqnarray}
where $U_{h}=(U_{1h},\ldots, U_{dh}), h=1,\ldots,N-1, U_{0}=U_{N}=0$ (because of the boundary conditions).
The ode system (\ref{semidL1}) is completed with the initial values
$$u_{j}^{N}(0)\in \mathbb{R}^{N-1}, j=1,\ldots d,$$ from the values of the components of $u_{0}$ at the nodes $x_{j}$.  Formulas (\ref{GNI1})-(\ref{GNI4}) are general; in some particular cases (constant boundary conditions, $A$ or $B$ independent of $u^{N}$, etc) they can be simplified. The grid values of the derivatives can be computed from the Legendre differentiation matrix, and shows the equivalence with a collocation method, cf. \cite{CanutoHQZ1988} for details.

We now study the existence, uniqueness and convergence of the semidiscrete approximation defined by (\ref{ad31a}). Having in mind the steps of the proof  of Theorem 2.2 of \cite{EAAD20}, the presence here of new nonlinearities in the pseudo-parabolic part introduces some relevant differences.
\begin{theorem}
\label{theorem32} Let $u_{0}\in X^{1}$. For all $t\in [0,T]$, there is a 
unique solution $u^{N}(t)$ of (\ref{ad31a}), (\ref{ad31b}) such that
\begin{eqnarray}
||u^{N}||_{L^{\infty}(0,T,X^{1})}\leq C,\label{exist}
\end{eqnarray}
for some constant depending on $||u_{0}||_{1}$. 
Furthermore, let $m\geq 1$, and 
assume that $u_{0}\in H_{0}^{m}$, $A, B, G,\gamma\in 
C^{m}(X^{m})$. If $u$ 
is the solution of (\ref{eq:21}), then
\begin{eqnarray}
||u^{N}-u||_{L^{\infty}(0,T,X^{0})}\leq C N^{-m},\label{ad37c}
\end{eqnarray}
for some constant $C$ independent of $N$. 
If, in addition, the elements of $B$ have uniformly bounded derivatives, then
\begin{eqnarray}
||u^{N}-u||_{L^{\infty}(0,T,X^{1})}\leq C N^{1-m}.\label{ad37b}
\end{eqnarray}
\end{theorem}
\begin{proof}
We first prove the local existence by using similar arguments to those of the first part of the proof of Lemma \ref{lemma1}. We consider an orthonormal basis $\{W_{1},\ldots,W_{N-1}\}$ of $\mathbb{P}_{N}^{0}$ and write 
\begin{eqnarray}
&&u^{N}(x,t)=(u_{1}^{N}(x,t),\ldots,u_{d}^{N}(x,t))^{T},\nonumber\\
&&u_{j}^{N}(x,t)=\sum_{k=1}^{N-1}U_{jk}(t)W_{k}(x),\quad j=1,\ldots,d,\label{31*}
\end{eqnarray}
Then (\ref{ad31a})  leads to an equivalent ode system of the form
\begin{eqnarray}
U_{js}'(t)+\sum_{k=1}^{N-1}L_{A(u^{N})}(W_{k},W_{s})U_{jk}'(t)&=&\sum_{k=1}^{N-1}L_{B(u^{N})}(W_{k},W_{s})U_{jk}(t)\nonumber\\
&&+\sum_{k=1}^{N-1}\langle G'(u^{N})\partial_{x}W_{k},W_{s}\rangle U_{jk}(t)\nonumber\\
&&+
\langle \gamma(u^{N}),W_{s}\rangle,\label{32*}
\end{eqnarray} 
for $s=1,\ldots,N-1,\; j=1,\ldots,d$, the functional $L$ is defined from (\ref{bil5}), and where if
$$u_{0}=(u_{0}^{(1)},\ldots,u_{0}^{(d)})^{T}, P_{N}^{10}u_{0}^{(j)}=\sum_{k=1}^{N-1}u_{0k}^{(j)}W_{k}, j=1,\ldots,d,$$ then $U_{js}(0)=u_{0s}^{(j)}, j=1,\ldots,d, s=1,\ldots,N-1$. Due to (H1), (H2), and Remark \ref{remark1}, standard ode theory can be applied to (\ref{32*}), defining a solution (\ref{31*}) locally in $t$. Existence of $u^{N}$ for all $t\in [0,T]$ is derived as follows. Let $t\in [0,T]$. By evaluating (\ref{ad31a}) at $\psi=u_{t}^{N}$ and using (H1), (H2), and Remark \ref{remark1}, similar arguments to those used in  Lemma \ref{lemma1} and Theorem \ref{theorem21} lead to
\begin{eqnarray*}
||u_{t}^{N}||_{1}\leq C ||u^{N}||_{1},
\end{eqnarray*}
for some constant $C$. Then, integrating over an interval
$(0,t)\subset [0,T]$ leads to
\begin{eqnarray*}
||u^{N}||_{1} = \left\|u^{N}(0)+\int_{0}^{t}u_{t}^{N}(s)ds\right\|_{1} 
\leq  ||u^{N}(0)||_{1}+C\int_{0}^{t}||u^{N}(s)||_{1}ds.\label{ad_226}
\end{eqnarray*} 
From the properties of the orthogonal projection we have $||u^{N}(0)||_{1}\leq C||u_{0}||_{1}$. 
Gronwall's lemma implies
the existence of $u^{N}(t)$ for all $t\in [0,T]$ 
and (\ref{exist}).

We now prove the error estimates (\ref{ad37c}) and (\ref{ad37b}). Let $\overline{u}$ be the projection of the solution $u$ of (\ref{eq:21}) with respect to (\ref{ad27e}) with $C=A, U=u^{N}$, and define
\begin{eqnarray*}
\eta:=\overline{u}-u, e^{N}:=u^{N}-u, 
\xi^{N}:=\overline{u}-u^{N}=\eta-e^{N}\in\mathbb{P}_{N}^{0}.\label{var}
\end{eqnarray*}

From (\ref{eq:21}) and (\ref{ad31a}) we have, for $\psi\in (\mathbb{P}_{N}^{0})^{d}$
\begin{eqnarray}
\mathcal{A}_{u^{N}}(u_{t}^{N},\psi)-\mathcal{A}_{u}(u_{t},\psi)=\mathcal{B}_{u^{N}}(u^{N},\psi)-\mathcal{B}_{u}(u,\psi),\label{des0}
\end{eqnarray}
Note that due to (\ref{ad27ee}), the left-hand side of (\ref{des0}) has the form
\begin{eqnarray*}
\mathcal{A}_{u^{N}}(u_{t}^{N},\psi)-\mathcal{A}_{u}(u_{t},\psi)&=&
\mathcal{A}_{u^{N}}(u_{t}^{N}-u_{t},\psi)+\mathcal{A}_{u^{N}}(u_{t},\psi)-\mathcal{A}_{u}(u_{t},\psi)\\
&=&\mathcal{A}_{u^{N}}(\xi_{t}^{N},\psi)+\mathcal{A}_{u^{N}}(u_{t},\psi)-\mathcal{A}_{u}(u_{t},\psi).
\end{eqnarray*}
Therefore, (\ref{des0}) can be written as
\begin{eqnarray}
\langle \xi_{t}^{N},\psi\rangle+\underbrace{\int_{\Omega}A(u^{N})\xi_{tx}^{N}\cdot\psi_{x}dx}_{J_{0}}&=&\underbrace{\int_{\Omega}\left(A(u)-A(u^{N})\right)u_{tx}\cdot\psi_{x}dx}_{J_{1}}
\nonumber\\
&&+\underbrace{\int_{\Omega} \left(B(u^{N})-B(u)\right)u_{x}^{N}\cdot\psi_{x}dx}_{J_{2}}\nonumber\\
&&+
\underbrace{\int_{\Omega} B(u)e_{x}^{N}\cdot\psi_{x}dx}_{J_{3}}\nonumber\\
&&+\underbrace{\int_{\Omega}\partial_{x}\left(G(u^{N})-G(u)\right)\cdot\psi dx}_{J_{4}}\nonumber\\
&&+\underbrace{\int_{\Omega}\left(\gamma(u^{N})-\gamma(u)\right)\cdot\psi dx}_{J_{5}}.\label{des1}
\end{eqnarray}
We now estimate each of the integrals of (\ref{des1}). Note that from the hypothesis (H1), Remark \ref{remark1}, and (\ref{bound1}) in Theorem \ref{theorem21},  we have
$$|J_{1}|\leq C ||e^{N}||_{1}||\psi||_{1}.$$ Similarly, from hypothesis (H2) and Remark \ref{remark1}
$$|J_{k}|\leq C ||e^{N}||_{1}||\psi||_{1}, \quad k=2,3,\quad
|J_{k}|\leq C ||e^{N}||_{0}||\psi||_{0},\quad k=4,5.$$
On the other hand, when (\ref{des1}) is evaluated at $\psi=\xi_{t}^{N}$ then
$$|J_{0}|\geq \alpha ||\xi_{tx}^{N}||_{0}.$$
All this above, when applied to (\ref{des1}) with $\psi=\xi_{t}^{N}$ leads to
\begin{eqnarray*}
||\xi_{t}^{N}||_{1}\leq C ||e^{N}||_{1}.
\label{ad_229}
\end{eqnarray*}
Since $e^{N}=\eta-\xi^{N}$ it holds that
\begin{eqnarray*}
||\xi^{N}(t)||_{1}=\left\|\int_{0}^{t}\xi_{t}^{N}(s)ds\right\|_{1}
\leq  C\int_{0}^{t}(||\xi^{N}(s)||_{1}+||\eta(s)||_{1})ds.\label{ad312}
\end{eqnarray*}
Note on the other hand that (\ref{ad31b}) implies that $u^{N}(0)=\overline{u}(0)$. Thus
$\xi^{N}(0)=0$.
Therefore, from Gronwall's lemma, 
the property $e^{N}=\eta-\xi^{N}$, (\ref{ad27b}), and Theorem \ref{theorem21}, (\ref{ad37b}) follows.

In order to prove the second estimate (\ref{ad37c}), we first apply Lax-Milgram 
theorem, \cite{Evans}, to ensure  the existence of 
$\varphi=(\varphi_{1},\ldots,\varphi_{d})^{T} \in X_{0}^{1}$ such that, \cite{BernardiM1989,CanutoHQZ1988}
\begin{eqnarray}
\mathcal{A}_{u^{N}}(\psi,\varphi)=\langle\xi_{t}^{N},\psi\rangle,\quad \psi\in X_{0}^{1}.\label{ad91}
\end{eqnarray}
with $\varphi\in X^{2}$ and
\begin{eqnarray}
||\varphi||_{2}\leq C ||\xi_{t}^{N}||_{0}.\label{ad92}
\end{eqnarray}
We evaluate (\ref{ad91}) at $\psi=\xi_{t}^{N}$ and write, cf. \cite{EAAD20}
\begin{eqnarray}
||\xi_{t}^{N}||_{0}^{2}= \mathcal{A}_{u^{N}}(\xi_{t}^{N},\varphi)=\mathcal{A}_{u^{N}}(\xi_{t}^{N},\varphi-P_{N}^{10}\varphi)+\mathcal{A}_{u^{N}}(\xi_{t}^{N},P_{N}^{10}\varphi).\label{ad93}
\end{eqnarray}
As for the first term on the right-hand side of (\ref{ad93}), note that hypothesis (H1), (\ref{c42c}), and (\ref{ad92}) imply
\begin{eqnarray*}
\mathcal{A}_{u^{N}}(\xi_{t}^{N},\varphi-P_{N}^{10}\varphi) &\leq & C||\xi_{t}^{N}||_{1}||\varphi-P_{N}^{10}\varphi||_{1} \leq  C N^{-1}||\xi_{t}^{N}||_{1}||\varphi||_{2}\\
& \leq & C N^{-1}||\xi_{t}^{N}||_{1}||\xi_{t}^{N}||_{0}.\label{ad94}
\end{eqnarray*}
On the other hand, using (\ref{ad27ee}), the second term is written as
\begin{eqnarray}
\mathcal{A}_{u^{N}}(\xi_{t}^{N},P_{N}^{10}\varphi)&=&-\mathcal{A}_{u^{N}}(e_{t}^{N},P_{N}^{10}\varphi)=-\mathcal{A}_{u^{N}}(u_{t}^{N},P_{N}^{10}\varphi)+
\mathcal{A}_{u^{N}}(u_{t},P_{N}^{10}\varphi)\nonumber\\
&=&-\mathcal{B}_{u^{N}}(u^{N},P_{N}^{10}\varphi)\nonumber\\
&&+\mathcal{A}_{u}(u_{t},P_{N}^{10}\varphi)+\mathcal{A}_{u^{N}}(u_{t},P_{N}^{10}\varphi)-\mathcal{A}_{u}(u_{t},P_{N}^{10}\varphi)\nonumber\\
&=&-\mathcal{B}_{u^{N}}(u^{N},P_{N}^{10}\varphi)+\mathcal{B}_{u}(u,P_{N}^{10}\varphi)
\nonumber\\
&&+\mathcal{A}_{u^{N}}(u_{t},P_{N}^{10}\varphi)-\mathcal{A}_{u}(u_{t},P_{N}^{10}\varphi)\nonumber\\
&=&\underbrace{-\mathcal{B}_{u^{N}}(u^{N},P_{N}^{10}\varphi)+\mathcal{B}_{u^{N}}(u,P_{N}^{10}\varphi)}_{K_{1}}\nonumber\\
&&-\underbrace{\mathcal{B}_{u^{N}}(u,P_{N}^{10}\varphi)+\mathcal{B}_{u}(u,P_{N}^{10}\varphi)}_{K_{2}}\nonumber\\
&&+\underbrace{\mathcal{A}_{u^{N}}(u_{t},P_{N}^{10}\varphi)-\mathcal{A}_{u}(u_{t},P_{N}^{10}\varphi)}_{K_{3}}.\label{des3}
\end{eqnarray}
We now estimate each of the differences $K_{i}$ in (\ref{des3}). The first $K_{1}$ is written as $K_{1}=K_{11}+K_{12}$ with
\begin{eqnarray*}
K_{11}&=&\int_{\Omega} B(u^{N})e_{x}^{N}\cdot(\varphi-P_{N}^{10}\varphi)_{x}dx+\langle\partial_{x}\left(G(u^{N})-G(u)\right),\varphi-P_{N}^{10}\varphi\rangle\\
&&+\langle\left(\gamma(u^{N})-\gamma(u)\right),\varphi-P_{N}^{10}\varphi\rangle\\
K_{12}&=&
-\int_{\Omega} B(u^{N})e_{x}^{N}\cdot\varphi_{x}dx+\langle\partial_{x}\left(G(u^{N})-G(u)\right),\varphi\rangle\\
&&+\langle\left(\gamma(u^{N})-\gamma(u)\right),\varphi\rangle.
\end{eqnarray*}
Observe that from hypothesis (H2), (\ref{c42c}), and (\ref{ad92}) we have
\begin{eqnarray}
|K_{11}|&\leq &C||e^{N}||_{1}||\varphi-P_{N}^{10}\varphi||_{1} \leq  C N^{-1}||e^{N}||_{1}||\varphi||_{2}\nonumber\\
& \leq  &C N^{-1}||e^{N}||_{1}||\xi_{t}^{N}||_{0}.\label{ad94b}
\end{eqnarray}
The estimate for the first term of $K_{12}$ makes use of the hypothesis that the elements of $B$ are bounded with bounded derivatives. Arguing by components, an integration by parts leads to integrals of the form
$$\int_{\Omega}e_{j}^{N}\left(b_{ij}(u^{N})\varphi_{ixx}+(b_{ij}'(u^{N})\cdot u_{x}^{N})\varphi_{ix}\right)dx,$$
and, along with hypothesis (H2) and (\ref{ad92}), $K_{12}$ can be bounded as
\begin{eqnarray}
|K_{12}|\leq C||e^{N}||_{0}||\varphi||_{2} \leq  C ||e^{N}||_{0}||\xi_{t}^{N}||_{0}.\label{ad94c}
\end{eqnarray}
On the other hand, note that $K_{2}$ can be written as
\begin{eqnarray}
K_{2}&=&\int_{\Omega} (B(u)-B(u^{N}))u_{x}\cdot(P_{N}^{10}\varphi)_{x}dx-\langle\partial_{x}\left(G(u^{N})-G(u)\right),P_{N}^{10}\varphi\rangle\nonumber\\
&&-\langle\left(\gamma(u^{N})\nonumber-\gamma(u)\right),P_{N}^{10}\varphi\rangle.\label{des4}
\end{eqnarray}
Therefore, the orthogonality properties of the projection $P_{N}^{10}$, Remark \ref{remark1}, (\ref{ad92}), and Theorem \ref{theorem21} imply that (\ref{des4}) can be estimated as
\begin{eqnarray}
|K_{2}|\leq C||e^{N}||_{0}||\varphi||_{2} \leq  C ||e^{N}||_{0}||\xi_{t}^{N}||_{0}.\label{ad94d}
\end{eqnarray}
Finally, $K_{3}$ is written as
\begin{eqnarray}
K_{3}=\int_{\Omega}\left(A(u^{N})-A(u)\right)u_{tx}\cdot(P_{N}^{10}\varphi)_{x}dx,\label{des5}
\end{eqnarray}
and again hypothesis (H1), orthogonality properties of the projection, and Theorem \ref{theorem21} are used to estimate (\ref{des5}) as
 \begin{eqnarray}
|K_{3}|\leq C||e^{N}||_{0}||\varphi||_{2} \leq  C ||e^{N}||_{0}||\xi_{t}^{N}||_{0}.\label{ad94e}
\end{eqnarray}
Now, applying (\ref{ad94}), (\ref{ad94b}), (\ref{ad94c}), (\ref{ad94d}), and (\ref{ad94e}) to (\ref{ad93}) it holds that
\begin{eqnarray*}
||\xi_{t}^{N}||_{0}\leq  CN^{-1}\left(||\xi_{t}^{N}||_{1}+||e^{N}||_{1}\right)
+C||e^{N}||_{0}.
\end{eqnarray*}
Thus, (\ref{ad37c}) follows from
(\ref{ad_229}), (\ref{ad37b}), 
$e^{N}=\eta-\xi^{N}$, Gronwall's lemma, and Theorem 
\ref{theorem21}. 
\end{proof}
\begin{remark}
\label{remark2}
The additional hypothesis on $B$ can be removed by using a 
similar argument to that in \cite{ArnoldDT1981}.
\end{remark}

\subsection{Full discretization with SSP methods}
\label{sec32}
As for the time discretization of the spectral ode semidiscrete systems from (\ref{semidL1}), we may consider, as in  \cite{EAAD20},  the singly diagonally implicit Runge-Kutta (SDIRK) 
methods of Butcher tableau
\begin{eqnarray}
\label{sdirk}
\begin{array}{c | cc}
\mu& \mu & 0  \\[2pt]
1-\mu& 1-2\mu& \mu\\[2pt]
\hline
\\[-9pt]
 & \frac{1}{2}  & \frac{1}{2}
 \end{array}
\end{eqnarray}
with $\mu=1/2$ (implicit midpoint rule, order 
two) and $\mu=\frac{3+\sqrt{3}}{6}$ (order three). Among other properties, the methods are A-stable (and therefore L-stable).

As explained in \cite{EAAD20}, methods like (\ref{sdirk}) are useful to prevent the possibility of oscillatory stiff phenomena in (\ref{semidL1}) or the corresponding spectral semidiscrete systems from nonregular data and when the hyperbolic terms are dominant. This is because of two main reasons: they are dispersive of high order (generating small dispersion errors of the oscillations in the numerical approximation, cf. e.~g. \cite{IzzoJ2021} and references therein for details) and they have the so-called strong stability preserving property, see e.~g. 
\cite{Gotlieb2005}. We observe that the SDIRK methods 
(\ref{sdirk}) are SSP methods and both were shown to be optimal 
(within the corresponding SDIRK schemes with the same stages and order). 

While the property of generation of small dispersive errors seems to be intrinsic to the temporal discretization, the SSP property somehow depends on the stability of full discretization of the semidiscrete system when this is numerically integrated in time with the forward Euler method. 
It may be worth analyzing this point in a more detailed way, cf. \cite{EAAD20}. We first consider the Legendre spectral Galerkin discretization of the scalar problem
\begin{equation}
\partial_{t}u+\partial_{x}f(u)=\epsilon \partial_{xx}u+\delta \partial_{xxt}u,\; x\in [-1,1], 0\leq t\leq T,\label{filt8}
\end{equation}
where $\epsilon,\delta>0$, along with homogeneous Dirichlet boundary conditions and initial data $u(x,0)=u_{0}(x)$. In (\ref{filt8}), the flux $f$ is assumed to be locally Lipschitz with $f(0)=0$. In the G-NI formulation, the corresponding system for the semidiscrete solution $u^{N}$ can be written in the form, \cite{EAAD20}
\begin{equation*}
(I_{N-1}-\delta\widetilde{D}_{N}^{(2)})\frac{d}{dt}\widetilde{U}(t)=\epsilon \widetilde{D}_{N}^{(2)}\widetilde{U}(t)+ \widetilde{D}_{N}^{(1)}\widetilde{f}(U(t)),\;t>0,\label{filt9}
\end{equation*}
where $U(t)=(U_{0}(t),\ldots,U_{N}(t))^{T}, U_{j}(t)=u^{N}(x_{j},t), j=0,\ldots,N$, $I_{N-1}$ is the $(N-1)\times (N-1)$ identity matrix, $D_{N}^{(1)}, D_{N}^{(2)}$ denote, respectively, the first- and second-derivative matrix at the Legendre-Gauss-Lobatto nodes $x_{j}, 0\leq j\leq N$; the computation of $f(U)$ is componentwise, and the tilde means that the first and last rows and columns (for matrices) and the first and last components (in the case of vectors) are removed due to the homogeneous boundary conditions. For a temporal discretization $t_{n}=n\Delta t, n=0,1,\ldots$, if $U_{FE}^{n}$ denotes the approximation to $U(t_{n})$ given by the forward Euler method, then
\begin{equation}
U_{FE}^{n+1}=(I_{N-1}+\epsilon\Delta t C_{N-1}^{-1}\widetilde{D}_{N}^{(2)})U_{FE}^{n}-\Delta t C_{N-1}\widetilde{D}_{N}^{(1)}\widetilde{f}(U_{FE}^{n}),\; n=0,1,\ldots,\label{filt10}
\end{equation}
with $U_{FE}^{0}=U(0)$ and where $C_{N-1}=I_{N-1}-\delta\widetilde{D}_{N}^{(2)}$ is assumed to be invertible, \cite{CanutoHQZ1988}. From the properties of $f$, we can find a constant $C=C(||U(0)||)$ such that
\begin{equation}
||\widetilde{f}(U_{FE}^{0})||\leq C||\widetilde{U}_{FE}^{0}||,\label{filt11}
\end{equation}
where $||\cdot ||$ denotes the Euclidean norm in $\mathbb{R}^{N-1}$. Using (\ref{filt11}) and induction on $n$ in (\ref{filt10}), a first approach shows that the SSP condition
\begin{equation*}
||U_{FE}^{n+1}||\leq ||U_{FE}^{n}||,\; n=0,1,\ldots,
\end{equation*}
holds when $\Delta t=\Delta t_{FE}$ satisfies
\begin{equation}
||I_{N-1}+\epsilon\Delta t C_{N-1}^{-1}\widetilde{D}_{N}^{(2)}||+\delta t C || C_{N-1}\widetilde{D}_{N}^{(1)}||\leq 1.\label{filt12}
\end{equation}
The extension of (\ref{filt12}) to systems of the form
\begin{equation*}
\partial_{t} u+\partial_{x}f(u)=B\partial_{xx}u+A\partial_{xxt}u,
\end{equation*}
for $u\in\mathbb{R}^{d}$, $A, B$ $d\times d$ matrices with $A$ positive definite, and $f$ locally Lipschitz with $f(0)=0$, is straightforward.

\subsubsection{Formulation of the fully discrete schemes}
The full discretization of (\ref{semidL1}) with (\ref{sdirk}) takes the following form. Let $F(U):=L(u^{N})U+\mathcal{H}(u^{N})$ and assume that $K(u^{N})$ is nonsingular. Let $\mathcal{F}(u^{N}):=K(u^{N})^{-1}F(u^{N})$. If $t_{n}=n\Delta t, n=0,1,\ldots$, then the numerical integration of (\ref{semidL1}) with the methods (\ref{sdirk}) has the form
\begin{eqnarray}
u^{*}&=&u^{n}+\mu\Delta t\mathcal{F}(u^{*}),\label{fullyd1}\\
u^{**}&=&u^{n}+(1-2\mu)\Delta t\mathcal{F}(u^{*})+\mu\Delta t\mathcal{F}(u^{**}),\label{fullyd2}\\
u^{n+1}&=&u^{n}+\frac{\Delta t}{2}\left(\mathcal{F}(u^{*})+\mathcal{F}(u^{**})\right),\nonumber
\end{eqnarray}
where $u^n$ denotes an approximation of the vector $U=(U_{jk}), j=1,2,\ldots,d,\; k=1,\ldots,N-1$, at $t=t_{n}$.
Each of the implicit systems in (\ref{fullyd1}), (\ref{fullyd2}) is solved with the classical fixed point algorithm
\begin{eqnarray}
u^{[\nu+1]}=\widetilde{u}^{n}+\mu\Delta t \mathcal{F}(u^{[\nu]}),\; \nu=0,1,\ldots,\label{fullyd4}
\end{eqnarray}
where $\widetilde{u}^{n}=u^{n}$ in the case of (\ref{fullyd1}) and $\widetilde{u}^{n}=u^{n}+(1-2\mu)\Delta t\mathcal{F}(u^{*})$ for (\ref{fullyd2}). The iteration (\ref{fullyd4}) is solved indeed as
\begin{eqnarray}
K(u^{[\nu]})X=\mu\Delta t F(u^{[\nu]}),\quad u^{[\nu+1]}=X+\widetilde{u}^{n},\; \nu=0,1,\ldots\label{fullyd5}
\end{eqnarray}
Due to the structure of the matrices $K^{(0)}$, $K^{(1)}$ and $K^{(2)}$, the system (\ref{fullyd5}) can be solved by $N\times N$ blocks. For example, in the case $d=2$, if (\ref{fullyd5}) is written in the form
\begin{eqnarray*}
\begin{pmatrix} K_{11}&K_{12}\\K_{21}&K_{22}\end{pmatrix}\begin{pmatrix}X_{1}\\X_{2}\end{pmatrix}=\mu\Delta t\begin{pmatrix}F_{1}\\F_{2}\end{pmatrix},
\end{eqnarray*}
where $K_{ij}\in \mathbb{R}^{N\times N}$, then the steps of the resolution may be as follows:
\begin{enumerate}
\item Factorization
\begin{eqnarray*}
\begin{pmatrix} K_{11}&K_{12}\\K_{21}&K_{22}\end{pmatrix}=\begin{pmatrix} K_{11}&0\\K_{21}&L_{22}\end{pmatrix}\begin{pmatrix} I&K_{11}^{-1}K_{12}\\0&U_{22}\end{pmatrix},
\end{eqnarray*}
with $L_{22}\in\mathbb{R}^{N\times N}$ lower triangular and $U_{22}\in\mathbb{R}^{N\times N}$ upper triangular. 
(We need $K_{11}$ to be nonsingular.)  Last system yields
\begin{eqnarray*}
K_{22}=K_{21}K_{11}^{-1}K_{12}+L_{22}U_{22},
\end{eqnarray*}
so $L_{22}, U_{22}$ come from the $LU$ factorization of $K_{22}-K_{21}\widetilde{K}$, where  $K_{11}\widetilde{K}=K_{12}$. 
\item Resolution of
\begin{eqnarray*}
\begin{pmatrix} K_{11}&0\\K_{21}&L_{22}\end{pmatrix}\begin{pmatrix}Y_{1}\\Y_{2}\end{pmatrix}=\mu\Delta t\begin{pmatrix}F_{1}\\F_{2}\end{pmatrix},
\end{eqnarray*}
that is $$K_{11}Y_{1}=\mu\Delta t F_{1},\quad L_{22}Y_{2}=\mu\Delta t F_{2}-K_{21}Y_{1}.$$
\item Resolution of
\begin{eqnarray*}
\begin{pmatrix} I&K_{11}^{-1}K_{12}\\0&U_{22}\end{pmatrix}\begin{pmatrix}X_{1}\\X_{2}\end{pmatrix}=\begin{pmatrix}Y_{1}\\Y_{2}\end{pmatrix},
\end{eqnarray*}
that is $$U_{22}X_{2}=Y_{2},\quad X_{1}=Y_{1}-K_{11}^{-1}K_{12}X_{2}.$$
\end{enumerate}
%

\section{Numerical experiments}
\label{sec4}
In this section we develop a computational study of the performance of the full discretization introduced and analyzed in section \ref{sec3}. The main goal of the numerical experiments is the illustration and investigation of several features: the spectral convergence of the semidiscretization revealed by Theorem \ref{theorem21} and the effects on the accuracy when some of the regularity hypotheses are lost. We will focus on the case $d=2$.
\subsection{Problem 1. Spectral convergence}
In order to illustrate the convergence of the method, we consider (\ref{eq:3psystem1a})-(\ref{eq:3psystem1c}) with $x_{L}=-\pi, x_{R}=\pi, A_{\pm}(t)=\pm \pi, B_{\pm}(t)=0$, and
\begin{eqnarray}
A=\begin{pmatrix} 2&1\\0&2\end{pmatrix},\quad
B=\begin{pmatrix} 1&0\\0&1\end{pmatrix},\quad
G(u,v)=\begin{pmatrix} uv\\u^{2}\end{pmatrix},\label{eq:41}
\end{eqnarray}
and $\gamma(u,v,x,t)$ such that the corresponding solution is given by
\begin{eqnarray}
u_{1}(x,t)=x+e^{-t}\sin{x},\quad u_{2}(x,t)=(1+t)\sin{x}.\label{eq:42}
\end{eqnarray}
Table \ref{t41} displays the corresponding errors at $T=1$ with $N=64$ and several time steps for the two time integrators.
\begin{table}[ht]
\begin{center}
\begin{tabular}{|c|c|c||c|c|}
    \hline
&    \multicolumn{2}{|c|} {$\mu=1/2$}& \multicolumn{2}{|c|}{$\mu=\frac{3+\sqrt{3}}{6}$}\\
\hline
{$\Delta t$} &{$L^{2}$ Error}&{$H^{1}$ Error}&{$L^{2}$ Error}&{$H^{1}$ Error}\\
\hline
0.1&1.2280E-03&3.2874E-03&6.3839E-05&1.8038E-04\\
0.05&3.0727E-04&8.2253E-04&8.0630E-06&2.2671E-05\\
0.0025&7.6833E-05&2.0567E-04&1.0119E-06&2.8395E-06\\
    \hline
\end{tabular}
\end{center}
\caption{Numerical approximation of (\ref{eq:3psystem1a})-(\ref{eq:3psystem1c}), (\ref{eq:41}): $L^{2}$ and $H^{1}$ 
norms of the error at $T=1$ with Legendre Galerkin 
method and $N=64$.}
\label{t41}
\end{table}
For this regular data, spectral accuracy in space is attained, since Table \ref{t41} shows the corresponding order of convergence of the time integrators.

A second example in this sense corresponds to taking
\begin{eqnarray}
A(u,v)=\begin{pmatrix} 4+u&0\\0&4+v\end{pmatrix},\quad
B(u,v)=\begin{pmatrix} u&v\\v&0\end{pmatrix},\quad
G(u,v)=\begin{pmatrix} uv\\u^{2}\end{pmatrix},\label{eq:43}
\end{eqnarray}
with the same boundary conditions as in the previous example and $\gamma$ such that the solution is also given by (\ref{eq:42}). As shown in Table \ref{t42}, a similar accuracy is observed. Note that the matrix $A$ in (\ref{eq:43}) does not satisfy in general the first condition in (H1), while $G$ is just locally Lipschitz, cf. (H2). This suggests that Theorem \ref{theorem32} may be valid under weaker conditions for the coefficients.

\begin{table}[ht]
\begin{center}
\begin{tabular}{|c|c|c||c|c|}
    \hline
&    \multicolumn{2}{|c|} {$\mu=1/2$}& \multicolumn{2}{|c|}{$\mu=\frac{3+\sqrt{3}}{6}$}\\
\hline
{$\Delta t$} &{$L^{2}$ Error}&{$H^{1}$ Error}&{$L^{2}$ Error}&{$H^{1}$ Error}\\
\hline
0.1&1.0270E-03&1.7528E-03&7.0755E-05&1.3089E-04\\
0.05&2.5684E-04&4.3823E-04&9.1976E-06&1.7529E-05\\
0.0025&6.4215E-05&1.0956E-04&1.1758E-06&2.2791E-06\\
    \hline
\end{tabular}
\end{center}
\caption{Numerical approximation of (\ref{eq:3psystem1a})-(\ref{eq:3psystem1c}), (\ref{eq:43}): $L^{2}$ and $H^{1}$ 
norms of the error at $T=1$ with Legendre Galerkin 
method and $N=64$.}
\label{t42}
\end{table}
\subsection{Problem 2. Nonsmooth data}
A second point concerns the performance of the method for nonregular data. As a first illustration, we consider the following linear example. With the matrices $A$ and $B$ from (\ref{eq:41}), the interval $[-1,1]$, $A_{pm}=B_{pm}=0$, and $G=\gamma=0$, the exact solution has the form
\begin{eqnarray}
u_{1}(x,t)=\sum_{n=1}^{\infty} e^{t\alpha_{n}}\left(U_{n}^{(1)}+t\beta_{n}U_{n}^{(2)}\right)X_{n}(x),\;
u_{2}(x,t)=\sum_{n=1}^{\infty} e^{t\alpha_{n}}U_{n}^{(2)}X_{n}(x),\label{eq:44}
\end{eqnarray}
where
\begin{eqnarray*}
&&X_{n}(x)=\sin\frac{n\pi}{2}(x+1),\quad n=1,2,\ldots, x\in [-1,1],\\
&&\alpha_{n}=\frac{-\lambda_{n}}{1-2\lambda_{n}},\quad \beta_{n}=-\alpha_{n}^{2},\quad \lambda_{n}=-\left(\frac{n\pi}{2}\right)^{2},\; n=1,2,\ldots,
\end{eqnarray*}
and
\begin{eqnarray*}
U_{j}(x)=\sum_{n=1}^{\infty} U_{n}^{(j)}X_{n}(x),\quad j=1,2,\quad x\in [-1,1].
\end{eqnarray*}
This allows to take nonsmooth initial data $U_{j}(x)$ and compare the numerical solution with a truncation of (\ref{eq:44}), considered as \lq exact\rq\ solution.

Thus, for the initial conditions
\begin{eqnarray}
U_{1}(x)=U_{2}(x)=\left\{\begin{matrix}1&|x|\leq 1/2\\0&{\rm otherwise}\end{matrix}\right.,\label{eq:45} 
\end{eqnarray}
and $\Delta t=h/2$, Table \ref{t43} shows the $L^{2}$ and $L^{\infty}$ errors of the two methods with respect to a truncated representation of (\ref{eq:44}), and for several values of $N$.
\begin{table}[ht]
\begin{center}
\begin{tabular}{|c|c|c||c|c|}
    \hline
&    \multicolumn{2}{|c|} {$\mu=1/2$}& \multicolumn{2}{|c|}{$\mu=\frac{3+\sqrt{3}}{6}$}\\
\hline
{$N$} &{$L^{2}$ error}&{$L^{\infty}$ error}&{$L^{2}$ error}&{$L^{\infty}$ error}\\
\hline
8&6.0707E-03&6.3989E-03&6.1858E-03&6.5514E-03\\
32&1.2539E-03&1.3093E-03&1.2601E-03&1.3210E-03\\
128&2.9529E-04&3.4544E-04&2.9565E-04&3.4619E-04\\
    \hline
\end{tabular}
\end{center}
\caption{Numerical approximation of (\ref{eq:3psystem1a})-(\ref{eq:3psystem1c}) from (\ref{eq:45}): 
$L^{2}$ and $L^{\infty}$ norms at $T=1$ of the error 
with Legendre Galerkin method and $\Delta t=0.5 h, h=2/N$. }
\label{t43}
\end{table}
As in the scalar case, \cite{EAAD20}, there is a dominant error in space, of order $O(N^{-1})$. Figure \ref{f41} shows the form of the components of the numerical solution at $T=1$.
\begin{figure}[ht!]
\centering
\subfigure[]
{\includegraphics[width=0.45\textwidth]{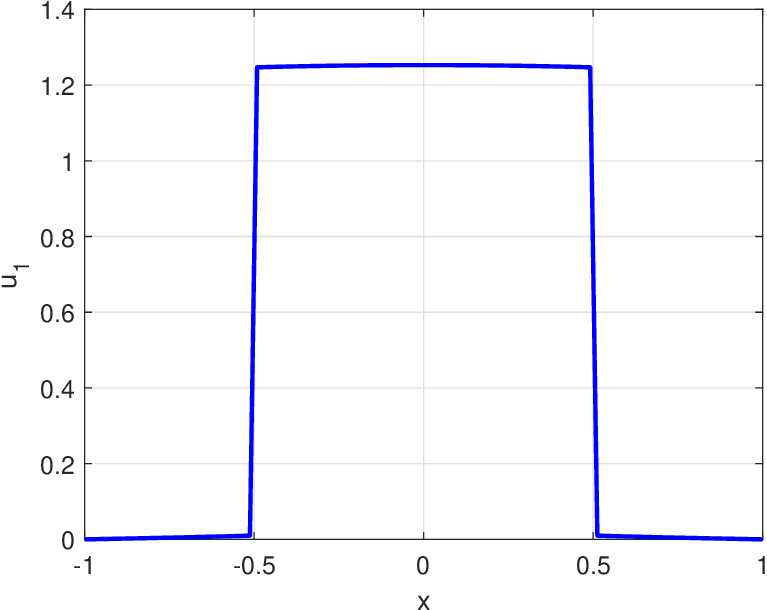}}
\subfigure[]
{\includegraphics[width=0.45\textwidth]{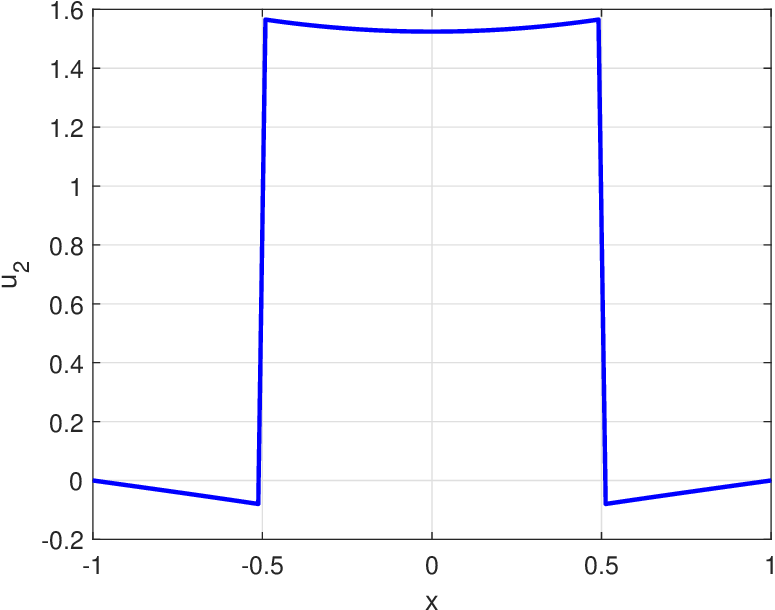}}
\caption{Numerical solution with Legendre Galerkin 
and SSP23 for the problem(\ref{eq:3psystem1a})-(\ref{eq:3psystem1c}) from (\ref{eq:45}) at $t=1$ with $\Delta t=0.025$.}
\label{f41}
\end{figure}


As a second example of the influence of the regularity, we consider the initial data
\begin{eqnarray}
U_{1}(x)=1-|x|,\quad U_{2}(x)=0.\label{eq:46}
\end{eqnarray}
\begin{table}[ht]
\begin{center}
\begin{tabular}{|c|c|c||c|c|}
    \hline
&    \multicolumn{2}{|c|} {$\mu=1/2$}& \multicolumn{2}{|c|}{$\mu=\frac{3+\sqrt{3}}{6}$}\\
\hline
{$N$} &{$L^{2}$ error}&{$L^{\infty}$ error}&{$L^{2}$ error}&{$L^{\infty}$ error}\\
\hline
16&6.6017E-04&9.6042E-04&6.8547E-04&1.0023E-03\\
32&1.6814E-04&2.5153E-04&1.7429E-04&2.6190E-04\\
64&4.2473-05&6.4026E-05&4.3990E-05&6.6600E-05\\
    \hline
\end{tabular}
\end{center}
\caption{Numerical approximation of (\ref{eq:3psystem1a})-(\ref{eq:3psystem1c}) from (\ref{eq:46}): 
$L^{2}$ and $L^{\infty}$ norms at $T=1$ of the error 
with Legendre Galerkin method and $\Delta t=0.5 h, h=2/N$. }
\label{t44}
\end{table}
In this case, the results from Table \ref{t44} seem to show an error in space of $O(N^{-2})$ (which, in the case of $\mu=1/2$ and since $\Delta t=O(h)$, coincides with the temporal error). The first component of the numerical solution at $T=1$ is shown in Figure \ref{f42}.
\begin{figure}[ht!]
\centering
{\includegraphics[width=0.55\textwidth]{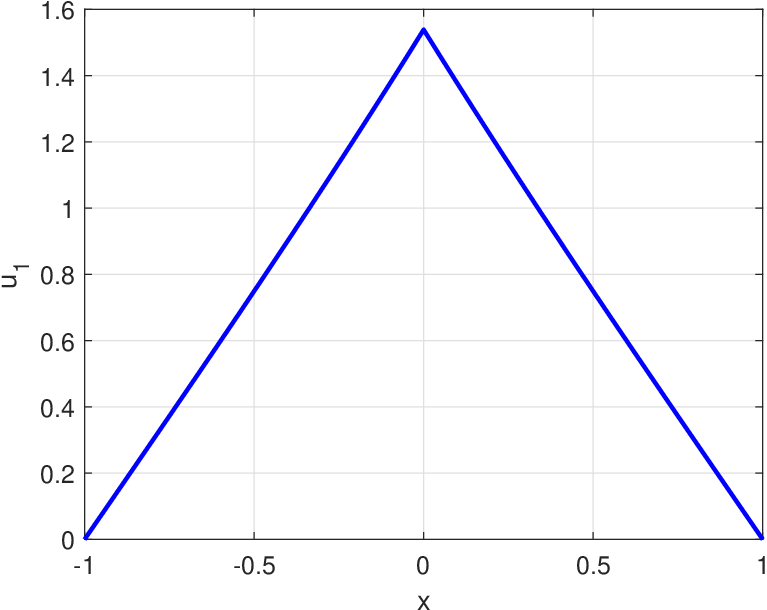}}
\caption{First component of numerical solution with Legendre Galerkin 
and SSP23 for the problem(\ref{eq:3psystem1a})-(\ref{eq:3psystem1c}) from (\ref{eq:46}) at $T=1$ with $\Delta t=0.025$.}
\label{f42}
\end{figure}
These two examples suggest that the reduction of order is proportionally related to the reduction of regularity of the data.

The last experiments are concerned with (\ref{eq:3psystem1a})-(\ref{eq:3psystem1c}) where $x_{L}=-56, x_{R}=200, G^{L}=g^{R}=\gamma=0, B=0$, 
\begin{eqnarray}
A(u,v)=\begin{pmatrix} \frac{1}{1+u^{2}}\\0&\frac{1}{1+v^{2}}\end{pmatrix},\label{eq:47}
\end{eqnarray}
initial condition $u_{0}=(U_{1},U_{2})^{T}$ with
\begin{eqnarray*}
U_{1}(x)=\left\{\begin{matrix}0.1&|x|\leq 0\\0&x>0\end{matrix}\right.,\quad
U_{2}(x)=\left\{\begin{matrix}0.9&|x|\leq 0\\0&x>0\end{matrix}\right.,
\end{eqnarray*}
and two flux functions: the one given in (\ref{eq:41}) and
\begin{eqnarray}
G(u,v)=\begin{pmatrix} \displaystyle\frac{u^{2}}{\lambda(u,v)}\\\displaystyle\frac{v^{2}}{\lambda(u,v)}\end{pmatrix},\quad \lambda(u,v)=\alpha v+(1-\alpha)v^{2}+u^{2}+(1-u-v)^{2},\label{eq:48}
\end{eqnarray}
with $\alpha=0.1$. The corresponding numerical approximation at $T=50$ is shown in Figures \ref{f43} and \ref{f44}. In the first case, each of the components seems to evolve to a structure whose main elements is some wave of dispersive shock type, traveling to the right. In the case of Figure \ref{f44}, the evolution of the initial discontinuity seems different, with the generation of wave structures of several type plus dispersion in both directions. In both experiments the numerical approximation does not seem to develop any kind of numerical artifact, since the dispersion presented seems to be part of the theoretical evolution, \cite{CongyEHS2021}.

\begin{figure}[ht!]
\centering
\subfigure[]
{\includegraphics[width=0.45\textwidth]{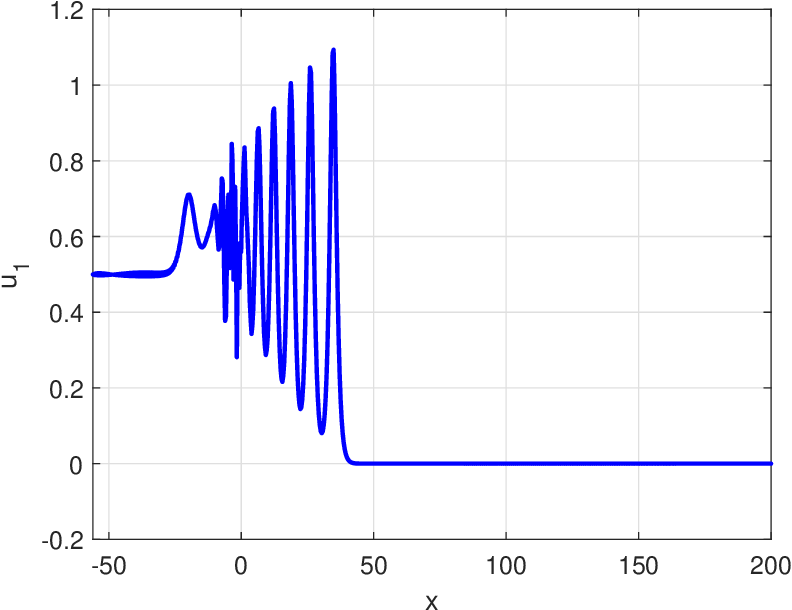}}
\subfigure[]
{\includegraphics[width=0.45\textwidth]{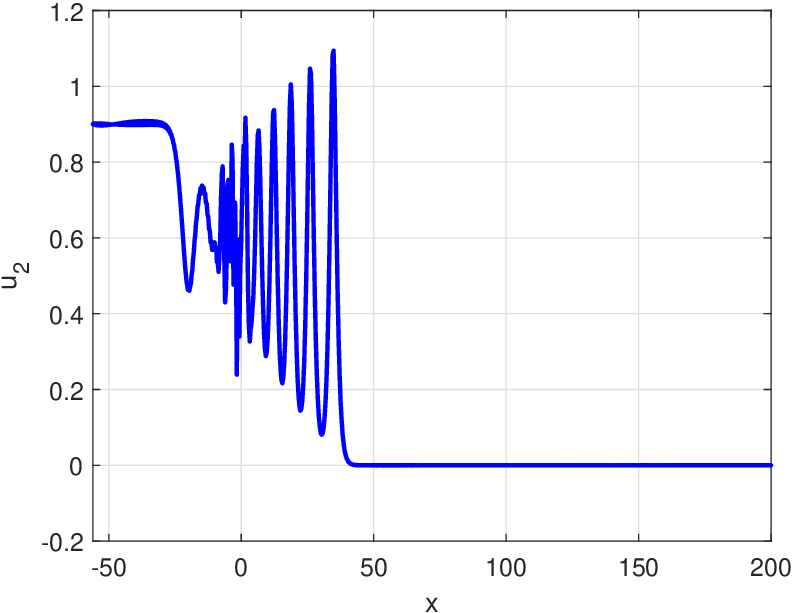}}
\caption{Numerical solution with Legendre Galerkin 
and SSP23 for the problem(\ref{eq:3psystem1a})-(\ref{eq:3psystem1c}) from (\ref{eq:47}) and $G$ given by (\ref{eq:41}) at $t=50$ with $\Delta t=0.025$.}
\label{f43}
\end{figure}
\begin{figure}[ht!]
\centering
\subfigure[]
{\includegraphics[width=0.45\textwidth]{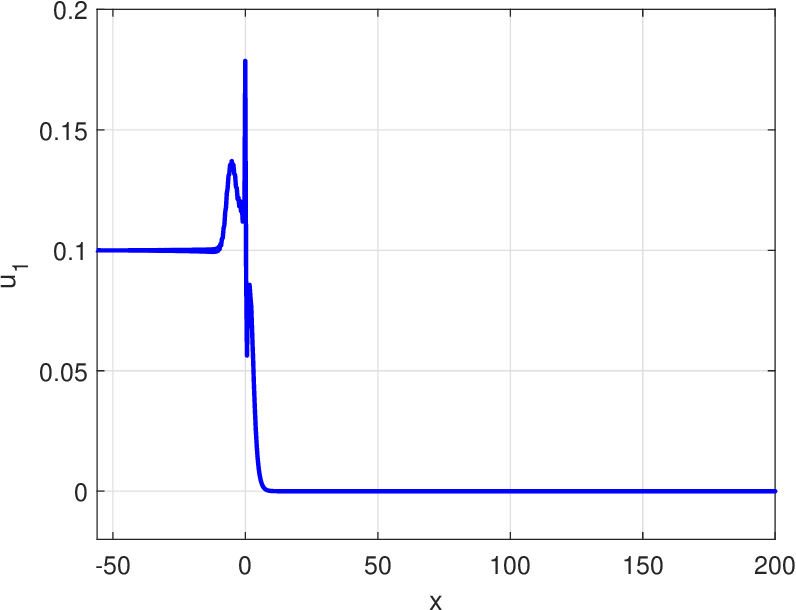}}
\subfigure[]
{\includegraphics[width=0.45\textwidth]{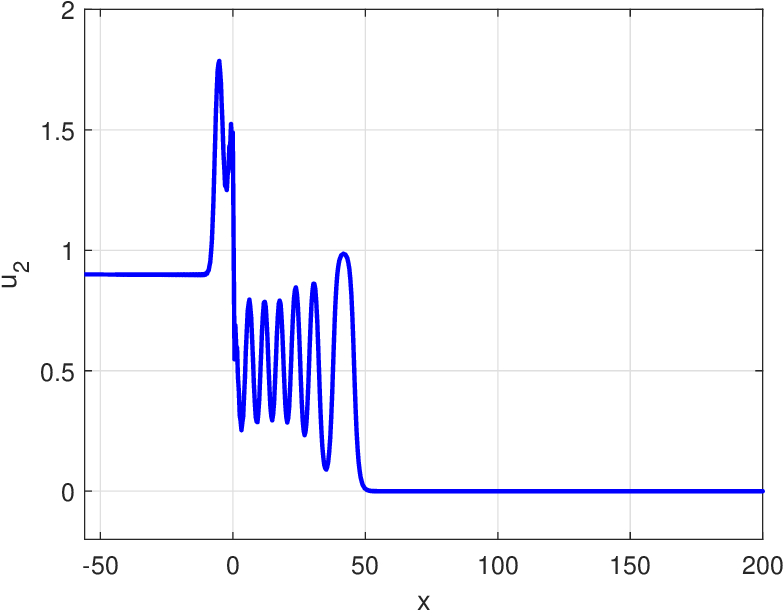}}
\caption{Numerical solution with Legendre Galerkin 
and SSP23 for the problem(\ref{eq:3psystem1a})-(\ref{eq:3psystem1c}) from (\ref{eq:48}) and $G$ given by (\ref{eq:48}) at $t=50$ with $\Delta t=0.025$.}
\label{f44}
\end{figure}

\section{Concluding remarks}
\label{sec5}
The present paper analyzes several aspects of the theory and numerical approximation of systems of pde's of htperbolic, pseudo-parabolic type. This kind of systems is characterized by the presence of terms of a combined character: hyperbolic, purely parabolic, and pseudo-parabolic, and then the models can be seen as diffusive-dispersive variants of conservation laws.

Focused on the one-dimensional ibvp with Dirichlet boundary conditions, the present study is divided into two parts. The first one introduces several mathematical properties: From a weak formulation of the problem, some results of well-posedness are proved. They include existence and uniqueness of solution, continuous dependence on the initial data, as well as a regularity result according to the level of smoothness of the elements of the equations. 

The second part of the paper is devoted to the numerical approximation. More specifically, the ibvp with Dirichlet boundary conditions is approximated in space with a spectral Galerkin discretization based on the Legendre polynomials. The semidiscrete approximation is shown to exist and two error estimates are proved. They depend on the degree of the polynomial approximation and the regularity of the solution. In particular, the smooth case leads to spectral convergence. On the other hand, the full discretization is completed with a temporal integration with strong stability character (SSP) and a high order of dispersion. Both properties are chosen to reduce that possible, spurious oscillatory behaviour in the simulation with nonregular data.

The performance of the resulting fully discrete scheme is computationally checked in a section of numerical experiments, with smooth and nonsmooth data. In the first case, spectral convergence is illustrated, even in examples where the elements, being smooth, do not satisfy some of the hypotheses required for the convergence result. This suggests that the error estimates still hold under less strict conditions. On the other hand, the experiments with nonsmooth data show an expected reduction of order, proportional to the decay of regularity. Furthermore, the evolution from some discontinuous initial conditions shows the formation of some dispersion which seems to be part of the solution and not some numerical artifact (of Gibbs type), since the stability seems to be mainly controlled by the properties of the time discretization.

The analysis and good performance shown in the present paper motivate us for a continuation of the work in several ways. The most immediate one is concerned with the application of the fully discrete scheme to the study of the dynamics of the three-phase hyperbolic, pseudo-parabolic transport system with non-equilibrium capillary pressures, which is currently in preparation, \cite{ACDL2024}. From a mathematical point of view, it is worth studying the extension of the convergence results of the spectral approach to semidiscretizations based on Jacobi polynomials associated to weights $w(x)=(1-x^{2})^{\mu}, -1<\mu<1$, which was considered in \cite{EAAD20}.  In particular, this would be useful to extend the spectral approach to ibvp's with another type of boundary conditions. The main point here is concerned with the comparison between the functionals
$$l(\varphi,\psi)=\int_{\Omega}\varphi_{x}(\psi w)_{x}dx,$$ (\cite{BernardiM1989,BernardiM1997})  and $$l_{a(u)}(\varphi,\psi)=\int_{\Omega}a(u)\varphi_{x}(\psi w)_{x}
dx,$$ for $u$ (fixed), $\varphi$ and $\psi$ in some weighted Sobolev space and some function $a=a(u)$ with suitable properties. Finally, a third line of future research consists of the extension of the results to the multi-dimensional case. Here, we think that the main drawback is computational, in the sense that the implementation will require the introduction of different tools to reduce the computational work, such as parallelization and dynamical low-rank approximation.

\section*{Acknowledgements}
The authors E. Abreu, A. Dur\'an and W. Lambert are supported by the 
Spanish Agencia Estatal de Investigaci\'on under Research Grant 
PID2020-113554GB-I00/AEI/10.13039/501100011033. E. Abreu is also supported by the
Brazilian National Council for Scientific and Technological Development (CNPq) (Grant No. 306385/2019-8) and
the State of S\~{a}o Paulo Research Foundation (FAPESP) (Grant No. 2022/15108-0). A. Dur\'an is also 
supported by the Junta de Castilla y Le\'on and FEDER funds (EU) 
under Research Grant VA193P20.

\appendix

\section{Traveling waves for the system (\ref{eq:3psystem1a}) from a conservation law}
\label{appen1}
The mathematical theory is completed here with a study on the behaviour of the systems when considered as a regularized variant of a conservation law.
It is a well-established fact that when we omit the diffusive and dispersive terms from Equation (\ref{eq:3psystem1a}), the resulting solution may display discontinuities within a finite time frame, even when the initial data is smooth. In such scenarios, it is imperative to analyze the solutions in their weak form, as detailed in the reference \cite{DAF21}. It is worth noting that these weak solutions are inherently non-unique, and, consequently, the application of specific criteria becomes essential for the purpose of distinguishing and selecting unique solutions. This process of selection is crucial in order to make informed decisions about which solutions are most relevant or suitable for the given context or problem at hand.

Various selection criteria are available for the purpose of identifying the physical solutions. These criteria encompass the vanishing viscosity, entropy conditions, traveling waves, and kinetic conditions.

In this context, we will provide a brief overview of the "traveling waves" criterion, which aids in the selection of these discontinuities. In classical problem scenarios, this criterion plays a significant role in determining which solutions align with the physical behavior of the system. For these cases, we consider only diffusive effects, i.e., 
\begin{equation}
\partial_{t}u+\partial_{x}f(u)=\epsilon(b(u)u_{x})_{x}.\label{cl1b}
\end{equation}
where $\epsilon>0$, and the diffusion coefficient $b=b(u)$ is smooth and bounded below by a positive constant.
The theory concerning traveling waves for this class of equations is thoroughly established. When the discontinuity allows for the existence of traveling waves, we refer to the equation represented by $(\ref{cl1b})$ as exhibiting a viscous profile. Numerous studies delve into the existence of these traveling waves. In classical scenarios, those that conform to Lax's or Liu's conditions are noteworthy, as documented in references such as \cite{DAF21,SMOL94}.

In these classical cases, the traveling waves are termed \lq compressive\rq. In other words, the characteristic waves converge upon the discontinuities from both sides. Nevertheless, the concept of traveling waves broadens the spectrum of admissible solutions. There exist discontinuities that do not meet the stringent criteria of the Lax or Liu conditions but still accommodate the presence of traveling waves.

For instance, we can mention transitional shocks known as "undercompressive shocks." In such cases, the characteristic waves only impinge in a single direction across the shock, as elaborated in references like \cite{Marc90,MP2001}. This expansion of the concept of traveling waves introduces a more nuanced understanding of the dynamics of discontinuities within these equations.

Nonetheless, in various models, the presence of dispersive terms holds significant importance in understanding the underlying physics of the problem. One illustrative case is the consideration of a diffusive-dispersive scalar conservation law, which can be represented in the following form:
\begin{equation}
\partial_{t}u+\partial_{x}f(u)=\epsilon(b(u)u_{x})_{x}+\delta(a(u)\partial_{x}u_{t})_{x},\label{cl1}
\end{equation}
where $\epsilon, \delta>0$, the diffusion coefficient $b=b(u)$ is smooth and bounded below by a positive constant, and $a$ is smooth and bounded above and below by positive constants. 

Equations of this nature have been used in numerous models where dispersive effects play a pivotal role. Notable among these is the classic paper by Benjamin, Bona, and Mahony, \cite{benjamin1972model}. Their work shed light on the relevance of such equations, highlighting their importance in various physical scenarios.

Moreover, this class of equations garnered substantial attention in the realm of hyperbolic models, thanks to the pioneering efforts of P. Lefloch. In his comprehensive book \cite{Lefloch}, Lefloch delves into the adaptation of hyperbolic equations to those incorporating dispersive terms. This extensive work encompasses the study of traveling waves and provides valuable insights into the dynamics of these systems. Lefloch's contributions have played a significant role in advancing our understanding of dispersive effects within hyperbolic models.

To illustrate the construction and acquisition of traveling waves, let's explore the following example in which we focus on the cubic flux function, $f(u) = u^3$, with $a$ and $b$ both set to 1. This simplification reduces Equation $(\ref{cl1})$ to the form:

\begin{equation}
\partial_{t}u + \partial_{x}u^{3} = \epsilon u_{xx} + \delta u_{xxt}. \label{cl2}
\end{equation}

In order to derive traveling waves, we assume that the parameters $\epsilon$ and $\delta$ are small and positive.
Then, we can identify three distinct regimes:

\begin{enumerate}
\item In the first regime, we assume that $\alpha$ and $\epsilon$ are of the same order of magnitude. In this case, we can choose $\alpha = \epsilon/\sqrt{\delta}$ as a constant. This configuration results in an equilibrium between diffusive and dispersive terms.
\item In the second regime, we consider the scenario in which $\epsilon$ is much smaller than $\delta$, and moreover, $\epsilon$ approaches zero more rapidly than $\delta$. Here, the diffusion effect is weaker compared to the pronounced influence of dispersion.
\item In the final case, we investigate when $\delta$ is significantly smaller than $\epsilon$, with $\delta$ diminishing at a faster rate than $\epsilon$. In this context, the dispersion effect is notably weaker compared to diffusion.
\end{enumerate}

These different regimes offer insights into the interplay between diffusive and dispersive terms and provide a framework for studying the behavior of traveling waves in this specific model.

\subsection{The case that $\alpha = \epsilon/\sqrt{\delta}$}
Building upon the work presented in \cite{Lefloch}, our focus is on exploring traveling wave solutions of (\ref{cl2}) that connect two distinct states, denoted as $u^-$ and $u^+$, within the phase space. To characterize these solutions, we introduce a traveling profile denoted as $u(y)$, and we consider a self-similar variable:

\begin{equation}
y = \alpha \frac{x - \lambda t}{\epsilon} = \frac{x - \lambda t}{\sqrt{\delta}}, \quad \lambda \neq 0. \label{uvar}
\end{equation}

By substituting $u = u(y)$ with $y$ defined as in  (\ref{uvar}), into (\ref{cl2}) and applying the chain rule, we derive that the profile $y \mapsto u(y)$ must satisfy

\begin{equation}
-\lambda u_{y} + \partial_{y}u^{3} = \alpha u_{yy} - \lambda u_{yyy}. \label{cl3}
\end{equation}

Assuming that the traveling waves connect two equilibria, $u^-$ and $u^+$, i.e.,

\begin{equation}
\lim_{y \rightarrow \pm\infty} u(y) = u_{\pm}, \quad \lim_{y \rightarrow \pm\infty} u_{y}(y) = \lim_{y \rightarrow \pm\infty} u_{yy}(y) = 0, \label{eq1d}
\end{equation}

we can integrate (\ref{cl3}), leading to
\begin{equation}
-\lambda (u(y) - u_{-}) + (u(y)^{3} - u_{-}^{3}) = \alpha u_{y}(y) - \lambda u_{yy}(y). \label{cl4}
\end{equation}
Taking $y\rightarrow\infty$ and  the conditions in (\ref{eq1d}), we have
\begin{equation}
\lambda=\frac{u_{+}^{3}-u_{-}^{3}}{u_{+}-u_{-}}=u_{-}^{2}+u_{-}u_{+}+u_{+}^{2}.\label{shocksp}
\end{equation}
It's worth noting that the parameter $\lambda$ signifies the discontinuity speed according to the Rankine-Hugoniot condition, which for a conservation law 
\[
\partial_t u+\partial_x f(u)=0, 
\]
the speed $\lambda$ of a discontinuity connecting two states $u^-$ and $u^+$ is given by:
\begin{equation*}
\lambda=\frac{f(u^+)-f(u^-)}{u^+-u^-}.
\end{equation*}
Notice then that (\ref{shocksp}) denotes the velocity at which the traveling wave propagates and it is the speed of the discontinuity. Equation (\ref{cl4}) outlines the essential features of traveling wave solutions that link two equilibrium states. This equation serves as the foundation for investigating the dynamics of these waves within the specific context of the problem. 

To study the profile, first we fix $u_{-} > 0$ and use the speed $\lambda$ as a parameter. The line passing through $(u_{-},u_{-}^{3})$ with slope $\lambda$ intersects $f(u)=u^{3}$ in the state

\[
u^{3} = \lambda (u-u_{-})+u_{-}^{3},
\]
at three distinct points: $u_{0} = u_{-}$ and the roots $u_{1}$ and $u_{2}$ of
\[
u^{2} + uu_{-} + u_{-}^{2} = \lambda,
\]
which are
\[
u_{1} = \frac{1}{2}(-u_{-}+\sqrt{\Delta_{1}}) \quad \text{and} \quad u_{2} = \frac{1}{2}(-u_{-}-\sqrt{\Delta_{1}}), \quad\text{ with}\quad \Delta_{1} = 4\lambda-3u_{-}^{2}.
\]

The roots are real and satisfy
\[
u_{2}<u_{1}<u_{0},
\]
when $\lambda\in (3u_{-}^{2}/4,3u_{-}^{2})$. 
Since $u_0=u_-$, notice  also that $u_{0}+u_{1}+u_{2}=0$.

Now, our  objective is to ascertain the trajectory of the function $u(y)$ that connects the equilibrium point $u(-\infty) = u_- = u_0$ at $-\infty$ with the equilibrium point at $\infty$. Furthermore, we aim to identify which point, either $u_1$ or $u_2$, the trajectory will ultimately converge towards as it approaches $\infty$. To achieve this, we define a variable $v=v(y)$ satisfying 
\[
\frac{du}{dy}=v
\]
and we reconfigure (\ref{cl4}) into a system of equations for analysis and computation, as
\begin{equation*}
\frac{d}{dy}\begin{pmatrix}u\\v\end{pmatrix}=K(u,v)=\begin{pmatrix}v\\\frac{\alpha}{\lambda}v+g(u,\lambda)-g(u_{-},\lambda)\end{pmatrix},\label{cl5}
\end{equation*}
where $g(u,\lambda)=u-u^{3}/\lambda$. 

Notice that $v=0$ and $u_0$, $u_1$ and $u_2$ are equilibria for the flux $K(u,v)$, i.e, the function $K$ vanishes at the three equilibria $(u_{j},0), j=0,1,2$. By linearization of $K(u,v)$ at the equilibria, the eigenvalues of the Jacobian of $K(u,v)$ at any point $(u,0)$ are given by
\begin{eqnarray*}
\mu_{\pm}(u)=\frac{1}{2}\left(\frac{\alpha}{\lambda}\pm\sqrt{\Delta}\right),\;
\Delta=\Delta(u)=\frac{\alpha^{2}}{\lambda^{2}}+4g_{u}(u,\lambda)=\frac{\alpha^{2}}{\lambda^{2}}+4\left(1-3\frac{u^{2}}{\lambda}\right).
\end{eqnarray*}
Note that for $\alpha$ small enough, it holds that $\Delta(u_{j})>0, j=0,1,2$ for $\lambda\in (3u_{-}^{2}/4,3u_{-}^{2})$, with the eigenvalue $\mu_{+}(u_{j})>0$ for $j=0,1,2$. The analysis of the sign of $\mu_{-}$ leads to
\[ 
\mu_{-}(u_{0})>0,\quad \mu_{-}(u_{1})<0\quad\text{ and }\quad \mu_{-}(u_{2})>0.
\]
  In this case, we can see from the linearization that $(u_{0},0)$ and $(u_{2},0)$ are unstable equilibria (repulsors), while $(u_{1},0)$ is a saddle point. 
 Thus, one can construct a trajectory connecting  $u_{0}=u_-$ at $-\infty$, with  $u_{1}$ at $\infty$. 
The connection between a repulsor and a saddle point exhibits the classical Lax profile, and the same applies when we have a connection between a saddle point and an attractor. In this scenario, the shockwave can be characterized as a compressive shock. 

 After some computations, and following a similar construction made in \cite{Lefloch} (Chapter III, sec. 2), this happens when
\begin{equation}
u_{0}>\frac{2}{3}\sqrt{\frac{2}{\lambda}}\alpha,\;
\quad \text{ and }\quad u_{1}=-u_{0}+\frac{1}{3}\sqrt{\frac{2}{\lambda}}\alpha,\label{cl6}
\end{equation}
and the trajectory has the explicit form
\begin{eqnarray*}
u(y)=\frac{\alpha}{3\sqrt{2\lambda}}-\left(u_{-}-\frac{\alpha}{3\sqrt{2\lambda}}\right){\rm tanh}\left(\left(u_{-}-\frac{\alpha}{3\sqrt{2\lambda}}\right)y\sqrt{2\lambda}\right).
\end{eqnarray*}
The case we are discussing here pertains to the convex case of order 3. A comprehensive analysis of non-convex flux, exemplified by $f(u) = u - u^3$, can be found in detail in \cite{SSS15}. For a comprehensive theory on traveling waves with dispersive terms, we also recommend referring to \cite{EHS17}. A similar investigation was undertaken in \cite{13DPP}, focusing on the flux of Buckley-Leverett type, which constitutes a non-convex flux with a single inflection point. In both of these studies, the authors delve into the existence of non-classical traveling waves, particularly those that connect saddle-saddle points. These discontinuities feature traveling wave profiles that deviate from classical norms. The researchers also explore the presence of traveling waves with non-monotonic profiles. A similar behavior is discussed in a different context in \cite{EBL15}.

The general case presents even greater complexity. It necessitates an exploration of the various intersections between shock curves and the flux term $f(u)$. It is often more practical to examine individual cases within specific situations. Additionally, one must consider the asymptotic behavior, especially regarding whether diffusive terms or dispersive terms dominate the dynamics.

\subsection{The case $\delta<<\epsilon$}
In the case that there is a diffusive dominance in our model, we consider a self-similar variable
\begin{equation}
\eta = \frac{x - \lambda t}{\epsilon}. \label{uvar2}
\end{equation}
By substituting $u = u(\eta)$ with $\eta$ defined as in Equation (\ref{uvar2}) into Equation (\ref{cl2}) and applying the chain rule, we derive that the profile
\begin{equation}
-\lambda u_{\eta} + \partial_{\eta}u^{3} =  u_{\eta\eta} -\varepsilon \lambda u_{\eta\eta\eta}, \
\quad \text{ where } \quad \varepsilon =\frac{\delta}{\epsilon^2}. \label{cl3e}
\end{equation}

Assume that the traveling waves connect two equilibria, $u_-$ and $u_+$ and satisfy (\ref{eq1d}). We can integrate (\ref{cl3e}) yielding
\begin{equation}
-\lambda (u(\eta) - u_{-}) + (u(\eta)^{3} - u_{-}^{3}) = u_{\eta}(\eta) - \varepsilon\lambda u_{\eta\eta}. \label{cl4eta}
\end{equation}
Taking $y\rightarrow\infty$ and  the conditions in (\ref{eq1d}), we have the same value for $\lambda$ given by (\ref{shocksp}).  

Now, we can study the asymptotic behavior by considering the following series for $u(\eta)$ as
\begin{equation}
u(\eta)=u_0(\eta)+\varepsilon u_1(\eta)+\varepsilon^2 u_2(\eta)+\cdots. \label{expan}
\end{equation}
By substituting $(\ref{expan})$ in $(\ref{cl4eta})$ we obtain
\begin{equation*}
-\lambda (u_0+\varepsilon u_1+\cdots - u_{-}) + ((u_0+\varepsilon u_1+\cdots)^{3} - u_{-}^{3}) = (u_0+\varepsilon u_1+\cdots)_{\eta} - \varepsilon\lambda (u_0+\varepsilon u_1+\cdots)_{\eta\eta}. \label{cl4n}
\end{equation*}
By collecting the corresponding orders, we have
\begin{align}
\mathbb{O}(0) \quad & -\lambda (u_0 - u_{-}) + (u_0^{3} - u_{-}^{3}) = (u_0)_{\eta},\label{oz}\\
\mathbb{O}(\epsilon) \quad & -\lambda u_1  + 3u_1u_0^{2}+ \lambda (u_0)_{\eta\eta}= (u_1)_{\eta},\notag\\
\mathbb{O}(\epsilon^2) \quad & -\lambda u_2  + 3u_0(u_1^2+u_2u_0)+ \lambda (u_1)_{\eta\eta}= (u_2)_{\eta},
\label{o2}\\
\vdots\notag
\end{align}
where we assume that
\begin{equation*}
\lim_{\eta\longrightarrow -\infty}u_i(\eta)=\lim_{\eta\longrightarrow \infty}u_i(\eta), \quad \text{ for } \quad i=1,2,\cdots.\label{eqcon}
\end{equation*}
Notice that equations are non-linear for $\mathbb{O}(0)$, however, they are linear for higher orders. 

To obtain the asymptotic series, we first analyse (\ref{oz}). The equilibria for this equation is obtained when 
\[
-\lambda (u_0 - u_{-}) + (u_0^{3} - u_{-}^{3})=0.
\]
We need to determine which solutions correspond to $u_0 = u_-$, $u_0 = u_+$, and $u_0 = -(u_- + u_+)$. Upon analysis, it becomes evident that the equilibrium $u_0 = -(u_- + u_+)$ is never attained. Furthermore, we observe that equation (\ref{oz}) admits a solution when $u_- > u_+$, representing the classical profile. In such a scenario, a monotone profile emerges. The correction is obtained for the different orders of $\varepsilon$. In Figure $\ref{figc}$, we show an example for $u_-=2$, $u_+=1$. For this case, $\lambda=7$. We show the $\mathbb{O}(0)$, $\mathbb{0}(0)+\mathbb{O}(\varepsilon)$ and $\mathbb{O}(0)+\mathbb{O}(\varepsilon)+\mathbb{O}(\varepsilon^2)$. 

\begin{figure}[ht!]
\centering
{\includegraphics[width=0.7\columnwidth]{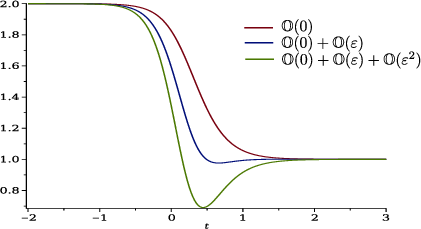}}
\caption{The solution of $(\ref{oz})$-$(\ref{o2})$ for $u_-=2$, $u_+=1$ and $\lambda=2$. If we take $\varepsilon$ to zero, the solution converges to the classical wave satisfying the Lax's condition. The higher order terms correct the solution for non-monotone and non-classical connections.}
\label{figc}
\end{figure}

\subsection{The case $\epsilon<\delta$}
We complete the analysis of existence of traveling waves by studying the case with dispersive dominance. Now we consider the self-similar variable
\begin{equation}
\eta = \frac{x - \lambda t}{\sqrt{\delta}}. \label{uvar3}
\end{equation}
By substituting $u = u(\eta)$ with $\eta$, defined as in Equation (\ref{uvar3}), into (\ref{cl2}), and applying the chain rule, we derive that the profile satisfies
\begin{equation}
-\lambda u_{\eta} + \partial_{\eta}u^{3} =  \varepsilon u_{\eta\eta} -\lambda u_{\eta\eta\eta}, \
\quad \text{ where } \quad \varepsilon =\frac{\epsilon}{\sqrt{\delta}}. \label{cl3ee}
\end{equation}

Assuming that the traveling waves connect two equilibria, $u_-$ and $u_+$ and satisfy (\ref{eq1d}), we can integrate  (\ref{cl3ee}) to have
\begin{equation}
-\lambda (u(\eta) - u_{-}) + (u(\eta)^{3} - u_{-}^{3}) = \varepsilon u_{\eta}(\eta) - \lambda u_{\eta\eta}. \label{cl4etae}
\end{equation}
Taking $y\rightarrow\infty$ and  the conditions in (\ref{eq1d}), we have the same value for $\lambda$ given by(\ref{shocksp}).  
Substituting now the expansion (\ref{expan}) into (\ref{cl4etae}) leads to
\begin{eqnarray}
&&-\lambda (u_0+\varepsilon u_1+\cdots - u_{-}) + ((u_0+\varepsilon u_1+\cdots)^{3} - u_{-}^{3})=\nonumber\\
&& \varepsilon(u_0+\varepsilon u_1+\cdots)_{\eta} - \lambda (u_0+\varepsilon u_1+\cdots)_{\eta\eta}. \label{cl4ne}
\end{eqnarray}
By collecting the corresponding orders, it holds that
\begin{align}
\mathbb{O}(0) \quad & -\lambda (u_0 - u_{-}) + (u_0^{3} - u_{-}^{3}) = -\lambda (u_0)_{\eta\eta},\label{oze}\\
\mathbb{O}(\epsilon) \quad & -\lambda u_1  + 3u_1u_0^{2}-(u_0)_{\eta\eta}= -\lambda(u_1)_{\eta},\notag\\
\mathbb{O}(\epsilon^2) \quad & -\lambda u_2  + 3u_0(u_1^2+u_2u_0)-(u_1)_{\eta\eta}= -\lambda(u_2)_{\eta}.\notag\\
\vdots\notag
\end{align}
If $\displaystyle{\frac{du_0}{d\eta}=v_0}$, then (\ref{oze}) can be written as a system
\begin{equation*}
\displaystyle{\frac{d}{d\eta}\begin{pmatrix}
u_0\\
v_0
\end{pmatrix}=
\begin{pmatrix}
v_0\\
\displaystyle{\frac{\lambda (u_0 - u_{-}) - (u_0^{3} - u_{-}^{3})}{\lambda}}
\end{pmatrix}},\label{dseo}
\end{equation*}
which admits equilibria $(0,u_{0})$, where
$u_0=u_-$, $u_0=u_+$, and $u_0=-(u_-+u+)$. By examining the eigenvalues of the Jacobian matrix of the flux, we derive two distinct eigenvalues:

\begin{equation}
\mu_- = -\sqrt{\frac{\lambda-3u_0^2}{\lambda}}\quad\text{ and }\quad \mu_+ = \sqrt{\frac{\lambda-3u_0^2}{\lambda}}.
\end{equation}

Observe that when $\lambda < 3u_0^2$, the eigenvalues are imaginary. In this scenario, the equilibria are identified as centers, and there is no possibility of connection between them. However, if $\lambda > 3u_0^2,$ the equilibrium becomes a saddle point. Nevertheless, it is evident that, for any combination of $u_-$ and $u_+$, it is impossible to satisfy the condition $\lambda > 3u_0^2$ simultaneously for two different equilibria. Consequently, in this regime, there is no solution for equation (\ref{oze}), and traveling waves do not exist. This observation aligns with the results presented in \cite{EHS17}, highlighting the importance of diffusive terms to support the profile in this specific case.

\end{document}